\documentclass[12pt,reqno]{amsart}
\usepackage{amsmath, amsthm, amssymb}
\usepackage[T1]{fontenc}
\usepackage[OT2, T1]{fontenc}
\usepackage[russian,english]{babel}
\usepackage{geometry}
\usepackage{cancel}
\usepackage{soul}
\usepackage{comment}
\usepackage{color}
\usepackage{amscd}
\usepackage{tabularx}
\usepackage{url}
\usepackage{eurosym}
\usepackage{braket}

\usepackage{amsrefs}
\usepackage{hyperref}

\usepackage{setspace}
\usepackage{verbatim}
\usepackage{mathrsfs}
\usepackage{dsfont}
\usepackage{amsbsy}
\usepackage{amstext}
\usepackage{amsthm}
\usepackage{amssymb}
\usepackage{xcolor}
\usepackage{cancel}
\usepackage{graphicx}
\usepackage[export]{adjustbox}

\usepackage{float}
\PassOptionsToPackage{normalem}{ulem}
\usepackage[toc,page]{appendix}
\usepackage{graphicx,tikz}

\usepackage{adjustbox,graphicx,subcaption}
\makeatletter
\usepackage{eurosym}
\usepackage{mathrsfs}
\usepackage{dsfont}
\usepackage{xcolor}
\usepackage{cancel}

\newtheorem{theorem}{Theorem}[section]
\newtheorem{lemma}[theorem]{Lemma}
\newtheorem{corollary}[theorem]{Corollary}
\newtheorem{conjecture}[theorem]{Conjecture}

\newtheorem{proposition}[theorem]{Proposition}

\theoremstyle{definition}
\newtheorem{definition}[theorem]{Definition}

\theoremstyle{remark}
\newtheorem{remark}[theorem]{Remark}

\numberwithin{equation}{section}

\def\bfa{{\mathbf a}}

\def\bfj{{\mathbf j}}

\def\bfp{{\mathbf p}}

\def\bfx{{\mathbf x}}
\def\bfy{{\mathbf y}}

\def \sU {\mathscr U}

\def\cH{{\mathcal H}}

\def \cH {\mathcal H}

\def \cM {\mathcal M}

\def \cS {\mathcal S}

\def \bzero {\mathbf 0}

\def \btau {\boldsymbol{\tau}}

\def \balp {{\boldsymbol{\alp}}}

\def \bgam {{\boldsymbol{\gam}}}
\def \bdel {{\boldsymbol{\del}}}

\def \bpsi {{\boldsymbol{\psi}}}
\def \bx {\mathbf{x}}
\def \by {\mathbf{y}}
\def \bz {\mathbf{z}}
\def \bj {\mathbf{j}}
\def \ba {\mathbf{a}}
\def \bob {\mathbf{b}}
\def \bop {\mathbf{p}}
\def \bk {\mathbf{k}}
\def \bv {\mathbf{v}}

\def\R{{\mathbb R}}
\def\Z{{\mathbb Z}}

\def \supp {{\mathrm{supp}}}

\def \ds1 {\mathds{1}}

\def\alp{{\alpha}} 
  
\def\gam{{\gamma}}

\def\del{{\delta}}

\def\le{\leqslant}

\def\Nwmh{{\mathfrak{N}_{w, \cM}(Q, \delta)}}
\def\Nwm{{\mathfrak{N}_{w, \cM}(Q, \bdel)}}
\def\DNwm{{{\mathfrak{N}}^{*, 1}_{w, \cM}(Q^*, \delta^*)}}
\def\DNwms{{{\mathfrak{N}}^{*, s}_{w, \cM}(Q^*, \delta^*)}}
\def\aq{{\left(\frac{\ba}{q}\right)}}
\def\aj{{\left(\frac{\ba}{j_1}\right)}}
\def\ajs{{\left(\frac{\ba}{j_s}\right)}}
\def\kq{{\left(\frac{\bk}{q}\right)}}
\def\kj{{\left(\frac{\bk}{j_1}\right)}}
\def\kjs{{\left(\frac{\bk}{j_s}\right)}}

\def\delp{{ \bdel^{\times}}}
\def\delpr{{ \bdel^{\times}_r}}
\def\delpone{{ \bdel^{\times}_1}}
\def\delps{{ \bdel^{\times}_s}}
\def\lgQ*{{\lceil\frac{\log 4Q^*}{\log 2}\rceil}}
\newcommand{\RH} {\rho_{{\mathrm{RH}}}}
\def\aR{{a(R-1)+R}}
\def\aRR{{(a+1)(R-1)+R}}
\def\aone{{a_1(R-1)+R}}

\def\atwo{{a_2(R-1)+R}}
\def\atwoR{{(a_2+1)(R-1)+R}}
\def\alst{{\alpha_{\textrm{st}}}}

\newcounter{@ToDo}
\newcommand{\todo@helper}[1]{%
	({\color{blue}TODO~\arabic{@ToDo}: {#1\@addpunct{.}}})%
}
\newcommand{\todo}[1]{\stepcounter{@ToDo}%
	\relax\ifmmode\text{\todo@helper{#1}}%
	\else\todo@helper{#1}\fi%
}
\newcounter{@cdo}
\newcommand{\cdo@helper}[1]{%
	({\color{red}CITE~\arabic{@cdo}: {#1\@addpunct{.}}})%
}
\newcommand{\cdo}[1]{\stepcounter{@cdo}%
	\relax\ifmmode\text{\cdo@helper{#1}}%
	\else\cdo@helper{#1}\fi%
}

\author{Rajula Srivastava}

\address{Mathematical Institute, University of Bonn, Endenicher Allee 60,
53115, Bonn, Germany, and
\newline Max Planck Institute for Mathematics, Vivatsgasse 7,
53111, Bonn,
Germany.}
\email{rajulas@math.uni-bonn.de}

\begin{document}

\title[Rational Points in Non-Isotropic Neighborhoods]
{Counting Rational Points In Non-Isotropic Neighborhoods of Manifolds}

\begin{abstract}
In this manuscript, we initiate the study of the number of rational points with bounded denominators, contained in a \textit{non-isotropic} $\delta_1\times\ldots\times \delta_R$ neighborhood of a compact submanifold 
$\cM$ of codimension $R$ in $\mathbb{R}^{M}$. We establish an upper bound for this counting function which holds when $\cM$ satisfies a strong curvature condition, first introduced by Schindler-Yamagishi in \cite{schindler2022density}. Further, even in the isotropic case when $\delta_1=\ldots=\delta_R=\delta$, we obtain an asymptotic formula which holds beyond the range of distance to $\cM$ established in \cite{schindler2022density}.
Our result is also a generalization of the work of J.J. Huang \cite{huangduke} for hypersurfaces.

As an application, we establish for the first time an upper bound for the Hausdorff dimension of the set of weighted simultaneously well approximable points on a manifold $\cM$ satisfying the strong curvature condition, which agrees with the lower bound obtained by Allen-Wang in \cite{allen2022note}. Moreover, for $R>1$, we obtain a new upper bound for the number of rational points \textit{on} $\cM$, which goes beyond the bound in an analogue of Serre's dimension growth conjecture for submanifolds of $\mathbb{R}^M$ .

\end{abstract}

\maketitle

\section{Introduction}
\label{sec introduction}
The aim of this manuscript is to establish an asymptotic formula for the number of rational points close to smooth manifolds satisfying a certain geometric condition first studied in \cite{schindler2022density}. In the case when the codimension of such a manifold is greater than one, we establish, for the first time, upper bounds on the number of rational points in a nonisotropic neighborhood of the manifold. Further, even in the isotropic case, we extend the main result of \cite{schindler2022density} (see Theorem \ref{thm SY}), by obtaining an asymptotic formula which counts rational points beyond the range of distance to the manifold established in \cite{schindler2022density}.

Let $\cM$ be a bounded immersed submanifold of $\mathbb{R}^{M}$ with boundary, of dimension $n$ and codimension $R$. For $Q\in \mathbb{Z}_{\geq 1}$ and $\delta\in (0, 1/2)$, we define the counting function
$$N_{\cM}(Q, \delta):=\#\{(\bfp, q)\in \mathbb{Z}^{M+1}: 1\leq q\leq Q, \textrm{ dist}(\cM, \bfp/q)\leq \delta/q\}.$$
Here $\textrm{dist}$ denotes the distance with respect to the $L^\infty$ norm on $\mathbb{R}^M$, that is,
$$\textrm{ dist}(\cM, \bfp/q):=\inf_{\bfx\in \cM} \|\bx-\bfp/q\|_{\infty}.$$
The study of rational points near manifolds has seen rapid development in the recent years. While the problem of obtaining precise asymptotics and upper bounds for $N_{\cM}(Q, \delta)$ is interesting in its own right, it is also closely related to questions in Diophantine approximation and the dimension growth problem for submanifolds of $\mathbb{R}^M$ (see \S \ref{subsec dio app} and Conjecture \ref{conj dgc man} further below). 

We have the trivial upper bound
$$N_{\cM}(Q, \delta)\leq c_{\cM} Q^{n+1},$$
with $c_{\cM}>0$ depending only on $\cM$.
Indeed, if $\cM$ is a (compact piece) of a rational hyperplane in $\mathbb{R}^M$, then the above estimate is the best we can hope for, as there exist constants $c_{\cM}', c_{\cM}$ depending only on $\cM$ so that
$$
    c_{\cM}'Q^{n+1}\leq N_{\cM}(Q, \delta)\leq c_{\cM} Q^{n+1}.
$$
However, if $\cM$ is curved in some sense, a probabilisitic heuristic suggests that 
\begin{equation}
    \label{eq sharp estimate}
    c_{\cM}'\delta^{R}Q^{n+1}\leq N_{\cM}(Q, \delta)\leq c_{\cM} \delta^{R}Q^{n+1}\,,
\end{equation}
for $\delta$ above a critical threshold depending on $Q$ and the codimension of $\cM$.
We are still far from understanding the precise curvature conditions that would be sufficient for the heuristic above to be true, in an appropriate range of $\delta$. However, the class of nondegenerate
manifolds, which arises quite frequently in Diophantine approximation, is a reasonable one to consider. Broadly speaking, a smooth (i.e., $C^{\infty}$) connected submanifold of $\mathbb{R}^M$
is nondegenerate if it is not contained in a proper affine subspace of $\mathbb{R}^M$. 

\begin{definition}
\label{def nondeg}
Let $\mathscr{U}, \mathscr{U}'$ be bounded open subsets of $\mathbb{R}^n$ with $\overline{\mathscr{U}}\subseteq\mathscr{U}'$. 
We say that an $l$-times continuously differentiable map 
$\Phi: \sU' \rightarrow \R^M$ is $l$-nondegenerate at a point 
$\mathbf{x}\in \sU'$ if the partial derivatives of $\Phi$ 
of order up to $l$ at the point $\mathbf{x}$ span $\R^M$. 
The map $\Phi$ is said to be $l$-nondegenerate if it is 
$l$-nondegenerate almost everywhere on $\sU'$ with respect 
to the $n$-dimensional Lebesgue measure.
We say that an immersed manifold $\cM:=\Phi(\mathscr{\overline{U}})$ is 
$l$-nondegenerate if $\Phi$ is $l$-nondegenerate.
\end{definition}

In the celebrated work \cite{Bers12}, Beresnevich established the lower bound in \eqref{eq sharp estimate} for \textit{analytic}, nondegenerate manifolds, in the range
$$\delta>Q^{-\frac{1}{R}}.$$ In the recent work \cite{DRN}, Schindler, Technau and the author proved the indicated lower bound for \textit{smooth}, nondegenerate manifolds  
in the range
$$\delta>Q^{-\frac{3}{2M-1}}.$$

\noindent Huxley, in \cite{huxley1994rational}, was the first to prove a near-optimal upper bound for sufficiently regular planar curves with non-vanishing curvature. This was followed by the remarkable work \cite{vauvel06} of Vaughan and Velani, in which they established the sharp result for such curves under a slightly stronger regularity assumption. A recent breakthrough came in \cite{huangduke}, where J.J. Huang proved an asymptotic for $N_{\cM}
(Q, \delta)$ when $\cM$ is a sufficiently smooth hypersurface with \textit{non-vanishing Gaussian curvature}, in the optimal range
\begin{equation}
    \label{eq hypersurface del}
    \delta>Q^{-1+\epsilon}.
\end{equation}
Further, in \cite{huang2024extremal}, he made the following conjecture for submanifolds $\cM$ of $\mathbb{R}^M$ of arbitrary dimension, satisfying the aforementioned nondegeneracy condition. 
\begin{conjecture}[\cite{huang2024extremal}, Conjecture 3.1]
\label{conj main}
Let $\cM$ be a bounded immersed submanifold of $\mathbb{R}^{M}$ with boundary, of dimension $n$ and codimension $R$. Suppose that $\cM$ is $l$-nondegenerate everywhere with $l\leq R+1$. Then there exists a constant $c_{\cM}>0$ depending only on $\cM$ such that
$$N_{\cM}(Q, \delta)\leq c_{\cM}\delta^R Q^{n+1},$$
when $\delta\geq Q^{-\frac{1}{R}+\epsilon}$ for some $\epsilon>0$ and $Q\to \infty$.
\end{conjecture}
The main theorems in \cite{vauvel06, huangduke} demonstrate that non-vanishing Gaussian curvature is sufficient to establish Conjecture \ref{conj main} for hypersurfaces. However, in \cite{srivastava2023density}, Technau and the author showed that the conjecture is also true for certain hypersurfaces with Gaussian curvature vanishing at a single point, provided the ``degree of flatness'' is below a critical value depending only on the dimension of the hypersurface. Further, when the degree of flatness is large, \cite{srivastava2023density} establishes a new asymptotic for $N_{\cM}(Q, \delta)$ incorporating the contribution due to the ``local flatness''.

For smooth, nondegenerate manifolds $\cM$ of arbitrary dimension, 
the current best upper bounds and asymptotics for a smoothened version of $N_{\cM}(Q, \delta)$, in terms of the range of $\delta$, are contained in \cite[Theorem 1.4 and 1.6]{DRN}. We also refer the reader to \cite{BY} for a previous result on upper bounds. 

However, all of these results on upper bounds and asymptotics remain valid only within the range of $\delta$ prescribed by Conjecture \ref{conj main}. It therefore came as a surprise when in \cite{schindler2022density}, Schindler and Yamagishi established an asymptotic for $N_{\cM}(Q, \delta)$ for manifolds $\cM$ satisfying a strong curvature condition,
in a range of $\delta$ which goes beyond Conjecture \ref{conj main} when the codimension is bigger than one! To state their result (and later ours) precisely, we first need some basic set-up. 
 
Recall that $\cM$ is a bounded immersed submanifold of $\mathbb{R}^{M}$ with boundary, of dimension $n$ and codimension $R$.
Since $\cM$ is compact, we can work locally. Using the implicit function theorem, we may assume without loss of generality that $\cM$ has the parametrization
\begin{equation}
    \label{eq manifold}
    \mathcal{M}:=\{(\bx, f_1(\bx), \ldots, f_R(\bx))\in \mathbb{R}^{n+R}: \bx\in \overline{B_{\varepsilon_0}(\bx_0)}\}.
\end{equation}
Here $\bx_0\in \mathbb{R}^n$, $f_r: \mathbb{R}^n\to \mathbb{R}$ are $C^{\infty}$ functions for $1\leq r\leq R$ and $\overline{B_{\varepsilon_0}(\bx_0)}$ denotes the closed ball in $\mathbb{R}^n$ centred at $\bx_0$ and of small enough radius $\varepsilon_0$. 

Further, we assume that $\cM$ satisfies the following 

\textbf{Curvature Condition:} Given any $\mathbf{t}=(t_1, \ldots, t_R)\in \mathbb{R}^R\setminus\{\mathbf{0}\}$, there exists a constant $C_{\mathbf{t}}>0$ such that
\begin{equation}
    \label{eq curv cond}
    \tag{CC}
    \min_{\bx\in \overline{B_{2\varepsilon_0}(\bx_0)}}\left |\det\, H_{\sum_{i=1}^R t_if_i}(\bx)\right|> C_{\mathbf{t}}.
\end{equation}
Note that when $R=1$, condition \eqref{eq curv cond} reduces to $\det H_{f_1} (\mathbf{x}_0) \neq 0$, which in turn is equivalent to non-vanishing Gaussian curvature for hypersurfaces. 
The main result of \cite{schindler2022density} is the following.

\begin{theorem}[\cite{schindler2022density}, Corollary 1.3]
\label{thm SY}
Let $\cM$ be as in \eqref{eq manifold} and let $n\geq 2$. Suppose Condition \eqref{eq curv cond} holds
and that $\varepsilon_0 > 0$ is sufficiently small. Then there exists a constant $c_{\cM} > 0$ depending only on $\cM$ such that
$$
N_{\cM}(Q, \delta) \sim  c_{\cM} \delta^{R} Q^{n +1}
$$
when 
\begin{equation}
    \label{eq SY del range}
    \delta \geq Q^{ - \frac{n}{n + 2(R-1) } + \epsilon}
\end{equation}
for any $\epsilon > 0$ sufficiently small and $Q \rightarrow \infty.$
In particular, Conjecture \ref{conj main} holds in this case.   
\end{theorem}
Note that
$$Q^{-\frac{n}{n + 2(R-1)}}<Q^{-\frac{1}{R}},$$
whenever $R>1$. Consequently, Theorem \ref{thm SY} goes beyond the range of $\delta$ hypothesized in Conjecture \ref{conj main} for manifolds satisfying \eqref{eq curv cond} with codimension strictly bigger than one. We refer the reader to \cite[Section 7]{schindler2022density} for examples of such manifolds. For a generalization of Theorem \ref{thm SY} for manifolds satisfying a less restrictive curvature condition, see \cite{munkelt}.

\subsection{Main Results}
Following \cite{huangduke, schindler2022density, srivastava2023density}, we establish our results for a smoothened version of the counting function $N_{\cM}(Q, \delta)$.
The same asymptotic and bounds for $N_{\cM}(Q, \delta)$ then follow by approximating the characteristic function of the ball $\overline{B_{\varepsilon_0}(\bx_0)}$ by smooth weight functions as in \cite[Section 7]{huangduke}. Let $w:\mathbb{R}^{n}\to [0,1]$ be a smooth function with 
$$\supp\,w\subseteq B_{\varepsilon_0}(\bx_0).$$ For $\delta\in \left(0, \frac{1}{2}\right)$, we define
\begin{equation}
    \label{def iso main cfunc}
    \mathfrak{N}_{w, \mathcal{M}}(Q, \delta):=\sum_{\substack{\ba\in \mathbb{Z}^n\\1\leq q\leq Q\\ \|qf_1(\ba/q)\|\leq \delta\\\vdots\\\|qf_R(\ba/q)\|\leq \delta}}w\left(\frac{\ba}{q}\right).
\end{equation}
We also introduce the exponent
\begin{align}
    \label{def theta}
    \Theta:&=\max\left(\frac{n(n+R+1)}{n+2R}, n+1-\frac{nR}{n+2(R-1)-\frac{4}{n}}\right)\nonumber\\
    &=\begin{cases}
       \frac{n(n+R+1)}{n+2R}, & 1\leq R \leq 2.\\
       n+1-\frac{nR}{n+2(R-1)-\frac{4}{n}}, & R\geq 3.
    \end{cases}
\end{align}
Our first result establishes an asymptotic for $\mathfrak{N}_{w, \mathcal{M}}(Q, \delta)$.
\begin{theorem}
\label{thm homog main}
For $n\geq 2$ and a sufficiently small $\varepsilon_0>0$, let $\cM$ be as in \eqref{eq manifold}. Suppose condition \eqref{eq curv cond} holds. Let $\Theta$ be as defined in \eqref{def theta}. Then there exists a constant $C_{w, \mathcal{M}}>0$ (depending only on $w$ and $\mathcal{M}$) such that for all $Q\geq 1$ and $\delta\in (0, 1/2)$, we have
\begin{equation}
\label{eq homog main est}
\mathfrak{N}_{w, \mathcal{M}}(Q, \delta)=\frac{2\hat{w}(\bzero)}{n+1}\delta^R Q^{n+1}+C_{w, \cM}\left(\delta^{R-1} Q^{n+1+\frac{\Theta-(n+1)}{R}}\Tilde{\mathcal{E}}_n(Q)^{\frac{1}{R}}+Q^{\Theta}\mathcal{E}_n(Q)\right),
\end{equation}
where  
\begin{equation}
    \label{eq aux error}
    \mathcal{E}_n(Q)=\Bigg\{\begin{array}{lr}
        \exp(\mathfrak{c}_1\sqrt{\log 4Q}) , & \text{if } n=2, R=1\\
        (\log 4Q)^{\mathfrak{c}_2}, & \text{if } n\geq 3, R=1\\
        \exp\left(\mathfrak{c}_2(\log \log 4Q)^2\right) & \text{if } n\geq 3, R\geq 2
        \end{array},
\end{equation}
and
\begin{equation}
    \label{eq aux error2}
    \Tilde{\mathcal{E}}_n(Q)=\Bigg\{\begin{array}{lr}
        \exp(\mathfrak{c}_1\sqrt{\log 4Q}) , & \text{if } n=2\\
        (\log 4Q)^{\mathfrak{c}_2} 
        & \text{if } n\geq 3
        \end{array},
\end{equation}
for large enough constants $\mathfrak{c}_1, \mathfrak{c}_2>0$ depending only on $w$ and $\mathcal{M}$ 
which can be calculated from the proof.
\end{theorem}
By approximating the characteristic function of the ball $\overline{B_{\varepsilon_0}(\bx_0)}$ by smooth weight functions
from above and below using standard arguments, we obtain the following asymptotic for  ${N}_{\mathcal{M}}(Q, \delta)$. For $R>1$, this is valid in a range of $\delta$ beyond 
Theorem \ref{thm SY}, and in particular, the one in Conjecture \ref{conj main}.

\begin{corollary}
\label{cor main asump}
For $n\geq 2$ and a sufficiently small $\varepsilon_0>0$, let $\cM$ be as in \eqref{eq manifold}. Suppose condition \eqref{eq curv cond} holds. Let $\Theta$ be as defined in \eqref{def theta}. Then there exists a constant $c_{\mathcal{M}}>0$ (depending only on  $\mathcal{M}$) such that
\begin{equation}
\label{eq main asymp}
{N}_{\mathcal{M}}(Q, \delta)\sim c_{\mathcal{M}}\delta^R Q^{n+1} \end{equation}
whenever
\begin{equation}
    \label{eq delta range}
    \delta \geq Q^{\frac{\Theta-(n+1)}{R}+\epsilon}= \max\left(Q^{-\frac{n+2}{n+2R}+\epsilon}, Q^{-\frac{n}{n+2(R-1)-\frac{4}{n}}+\epsilon}\right)
\end{equation} 
for any sufficiently small $\epsilon>0$ and $Q\rightarrow \infty$.
\end{corollary}

When $\delta=0$, the term $\mathfrak{N}_{w, \cM}(Q, 0)$ counts the weighted number of rational points with denominator bounded by $Q$ lying \textit{on} the manifold $\cM$. Conjecture \ref{conj main} would imply 
\begin{equation}
\label{serrdim}
N_{\cM}(Q, 0)\ll N_{\cM}(Q, Q^{-\frac{1}{R}+\epsilon})\ll Q^{n + \epsilon R}
\end{equation}
for any $\epsilon > 0$ sufficiently small, whenever $\cM$ is an $n$-dimensional bounded immersed submanifold of $\mathbb{R}^{M}$ of codimension $R$, which is $l$-nondegenerate with $l\leq R+1$.
One can consider the above as an analogue of Serre's dimension growth conjecture, but for submanifolds of $\mathbb{R}^{M}$ of dimension $n$. The original formulation for irreducible projective varieties is stated below.

\begin{conjecture}[Dimension Growth Conjecture]\label{conj dgc}
Let $X \subseteq \mathbb{P}_{\mathbb{Q}}^{M-1}$ be an irreducible projective variety of degree at least two defined over $\mathbb{Q}$. Let $N_X(B)$ be the number of rational points on $X$ of naive height bounded by $B$. Then
$$N_X(B)\ll_X B^{\dim X} (\log B)^c$$
for some constant $c > 0$.
\end{conjecture}

In \cite{Sal3}, Salberger established a version of the above conjecture with $B^{\epsilon}$ in place of $(\log B)^c$. We refer the reader to \cite{castryck2020dimension, schindler2022density} and the references therein for an introduction to the topic, and for further refinements. In general, the upper bound in Conjecture \ref{conj dgc} is sharp; for example, if $X$ contains a rational linear divisor. However, by excluding divisors of small degree and imposing stronger conditions (say on the degree), it is possible to obtain a better upper bound; for instance, in \cite{marmon2010density}, this has been achieved for hypersurfaces of degree at least four.

In view of Conjecture \ref{conj main} and the previous discussion, it is reasonable to formulate the following analogue of the dimension growth conjecture for the class of nondegenerate submanifolds of $\mathbb{R}^M$ (see 
Definition \ref{def nondeg}).

\begin{conjecture}[A Dimension Growth Conjecture for Nondegenerate Manifolds]
\label{conj dgc man}
Let $\cM$ be a bounded immersed submanifold of $\mathbb{R}^{M}$ with boundary, of dimension $n$ and codimension $R$. Further, suppose that $\cM$ is $l$-nondegenerate for $l\leq R+1$. Then there exists a constant $c_{\cM}>0$ depending only on $\cM$ such that
$$N_{\cM}(Q, 0)\leq c_{\cM}Q^{n+\epsilon},$$
for some $\epsilon>0$ and all $Q\geq 1$.
\end{conjecture}

The nondegeneracy condition, which implies that $\cM$ is not contained in a proper affine subspace of $\mathbb{R}^M$, can be considered to be a replacement for the requirement in Conjecture \ref{conj dgc} that the projective variety be irreducible and of degree at least two.

In \cite{schindler2022density}, as an immediate consequence of their main theorem, Schindler and Yamagishi obtained
that 
$$
{N}_{\cM}( Q, 0) \ll Q^{n -\frac{(n-2) (R-1) }{n + 2(R-1)} } (\log Q)^{c}$$
for some constant $c > 0$, whenever the compact manifold $\cM$ satisfies condition \eqref{eq curv cond}. In particular, for  submanifolds of $\mathbb{R}^M$ satisfying this much stronger curvature condition, their estimate broke the $Q^{n}$ barrier in Conjecture \ref{conj dgc man}.

In this paper, as a corollary of Theorem \ref{thm homog main}, we obtain the following improvement over \cite[Corollary 1.4]{schindler2022density},
which goes even further in pushing through the 
barrier in Conjecture \ref{conj dgc man} 
for submanifolds of $\mathbb{R}^M$ satisfying the curvature condition \eqref{eq curv cond}.

\begin{corollary}
    \label{cor Serre dim growth}
    For $n\geq 3$ and a sufficiently small $\varepsilon_0>0$, let $\cM$ be as in \eqref{eq manifold}. Suppose condition \eqref{eq curv cond} holds. Let $\Theta$ be as defined in \eqref{def theta}, and let $\mathcal{E}_n(Q)$ be as in \eqref{eq aux error}. Then
    $${N}_{\cM}(Q, 0)\ll Q^{\Theta} \mathcal{E}_n(Q)$$
    for all $Q\geq 1$, with the implicit depending only on $\cM$.
\end{corollary}

We now come to our main estimate, which is an upper bound for the number of rational points with bounded denominators contained in a non-isotropic neighborhood of a smooth manifold $\cM$ satisfying \eqref{eq curv cond}. It specializes to the upper bound in Theorem \ref{thm homog main} for isotropic neighborhoods of $\cM$. For $Q\in \mathbb{Z}_{\geq 1}$ and $\bdel=(\delta_1, \ldots, \delta_R)\in (0, 1/2)^R$, we define the counting function
$$N_{\cM}(Q, \bdel):=\#\left\{(\bfa, q)\in \mathbb{Z}^{n+1}: 1\leq q\leq Q,\, \left\|qf_r(\ba/q)\right\|\leq \delta_r/q \text{ for }1\leq r\leq R\right\},$$
and its smoothened version
\begin{equation}
    \label{def main cfunc}
    \mathfrak{N}_{w, \mathcal{M}}(Q, \bdel):=\sum_{\substack{\ba\in \mathbb{Z}^n\\1\leq q\leq Q\\ \|qf_1(\ba/q)\|\leq \delta_1\\\vdots\\\|qf_R(\ba/q)\|\leq \delta_R}}w\left(\frac{\ba}{q}\right).
\end{equation}
When $\delta_1=\delta_2\ldots=\delta_R=\delta$, we shall simply refer the above as $N_{\cM}(Q, \delta)$ and $\mathfrak{N}_{w, \mathcal{M}}(Q, \delta)$ respectively, so as to be consistent with the notations for the corresponding isotropic counting functions.

We again have the trivial upper bound
$$N_{\cM}(Q, \bdel)\leq c_{\cM} Q^{n+1}.$$
However, using a probabilistic heuristic, we expect that
$$c_{\cM}'\left(\prod_{r=1}^R \delta_r\right) Q^{n+1}\leq N_{\cM}(Q, \bdel)\leq c_{\cM} \left(\prod_{r=1}^R \delta_r\right)Q^{n+1}\,,$$
now for each $\delta_r$ (with $1\leq r\leq R$) above some critical threshold depending on $Q$ and the codimension $R$.
Let
\begin{equation}
    \label{def delp}
    \delp:= \prod_{r=1}^R\delta_r;
\end{equation}
and for $1\leq r\leq R$, set
\begin{equation}
    \label{def delpr}
    \delpr:= \prod_{\substack{1\leq s\leq R\\s\neq r}}\delta_s=\frac{\delp}{\delta_r}\,.
\end{equation}

Our main theorem is an upper bound for  $\mathfrak{N}_{w, \mathcal{M}}(Q, \bdel)$, which 
is the first ever non-trivial estimate for rational point count in a non-isotropic neigborhood of a submanifold of $\mathbb{R}^{M}$. 

\begin{theorem}[Main Theorem]
\label{thm main}
For $n\geq 2$ and a sufficiently small $\varepsilon_0>0$, let $\cM$ be as in \eqref{eq manifold}. Suppose condition \eqref{eq curv cond} holds. Let $\Theta$ be as defined in \eqref{def theta}. Then 
for all $Q\geq 1$ and $\bdel\in (0, 1/2)^R$, we have
\begin{align}
\label{eq main est}
\mathfrak{N}_{w, \mathcal{M}}(Q, \bdel)\ll \delp Q^{n+1} + \sum_{r=1}^R\delpr Q^{n+1+\frac{\Theta-(n+1)}{R}}\Tilde{\mathcal{E}}_n(Q)^{\frac{1}{R}}+Q^{\Theta}\mathcal{E}_n(Q),     \end{align}
where $\mathcal{E}_n(Q)$ and $\Tilde{\mathcal{E}}_n(Q)$ are as defined in \eqref{eq aux error} and \eqref{eq aux error2} respectively. The 
implicit constant depends only on $w$ and $\mathcal{M}$.
\end{theorem}

The following upper bound for $N_{\cM}(Q, \bdel)$ is a direct corollary of the above result.
\begin{corollary}
\label{cor hetero}
For $n\geq 2$ and a sufficiently small $\varepsilon_0>0$, let $\cM$ be as in \eqref{eq manifold}. Suppose condition \eqref{eq curv cond} holds. Let $\Theta$ be as defined in \eqref{def theta}. Then there exists a constant $c_{\mathcal{M}}>0$ (depending only on  $\mathcal{M}$) such that
\begin{equation}
\label{eq hetero asymp}
{N}_{\mathcal{M}}(Q, \bdel)\leq c_{ \mathcal{M}}\delp Q^{n+1}    
\end{equation}
whenever
\begin{equation}
    \label{eq delta hetero range}
    \min_{1\leq r\leq R}\delta_r\geq Q^{\frac{\Theta-(n+1)}{R}+\epsilon}= \max\left(Q^{-\frac{n+2}{n+2R}+\epsilon}, Q^{-\frac{n}{n+2(R-1)-\frac{4}{n}}+\epsilon}\right),
\end{equation} 
for any sufficiently small $\epsilon>0$ and $Q\rightarrow \infty$.
\end{corollary}

\subsection{Applications to Diophantine approximation}
\label{subsec dio app}
Next, we discuss applications of Theorem \ref{thm main} to Diophantine approximation on the manifold $\cM$. To do so, we need to introduce a definition and some notations. 
\begin{definition}
    \label{def psi approx}
    Given a family $\bpsi=(\psi_0, \psi_1, \ldots, \psi_R)$ of monotonic functions $\psi_r:(0,+\infty)\to(0,1)$ with $0\leq r\leq R$,  we call a point $\bfy\in\R^{n+R}$ {\em $\bpsi$-approximable} if the conditions
\begin{align}\label{psi-app}
 &\left|y_i-\frac{a_i}{q}\right|<\frac{\psi_0(q)}{q},\qquad 1\leq i\leq n,\\
&\left|y_{i}-\frac{a_{i}}{q}\right|<\frac{\psi_{i-n}(q)}{q},\qquad n+1\leq i\leq n+R, 
\end{align}
hold for infinitely many $(q, \bfa)=(q, a_1, \ldots, a_{n+R})\in\mathbb{Z}_{\geq 0}\times\Z^{M}$. 
\end{definition}
We shall denote the set of $\bpsi$-approximable points in $\R^{n+R}=\R^{M}$ by $\cS_{n+R}(\bpsi)$. For $\btau\in \mathbb{R}_{>0}^{1+R}$, given the approximation function family $\bpsi_{\btau}=\left(q^{-\tau_0}, q^{-\tau_1}, \ldots, q^{-\tau_R}\right)$, we shall abbreviate notation and just write $\cS_{n+R}(\btau):=\cS_{n+R}(\bpsi_{\btau})$.
We shall also call this the set of $\btau$-approximable points, or the set of $\btau$-weighted simultaneously well approximable points. By Dirichlet's theorem \cite{Schmidt-1980}, $\cS_{n+R}(1/n, 1/n, \ldots, 1/n)=\R^{M}$.

The weighted simultaneous approximation result below deals with the convergence case of Khintchine's theorem \cite{khintchine1926metrischen}. In the case when $\psi_0=\psi_1\ldots=\psi_R$, it complements the divergence case for \emph{analytic} nondegenerate submanifolds of $\R^{M}$ in \cite[Theorem~2.5]{Bers12}. 

\begin{theorem}
    \label{thm hausd well app}
    For $n\geq 2$ and a sufficiently small $\varepsilon_0>0$, let $\cM$ be as in \eqref{eq manifold}. Suppose condition \eqref{eq curv cond} holds. 
    Let $s>\frac{nR}{R+1}$ and let $\bpsi=(\psi_0, \psi_1, \ldots, \psi_R)$ be a family of monotonic approximation functions $\psi_r:(0,+\infty)\to(0,1)$ for $0\leq r\leq R$, with
    \begin{equation}
        \label{assump apprx func}
        \psi_0\leq \min \{\psi_1, \ldots, \psi_R\},
    \end{equation}
    and
    \begin{equation}
        \label{eq Hmeas gset conv}
        \sum_{q=1}^{\infty}q^{n}\left(\frac{\psi_0(q)}{q}\right)^{s}\prod_{r=1}^R\psi_r(q)<\infty.
    \end{equation}
    Then
    \begin{equation}
        \label{eq conc Hmeas zero}
        \mathcal{H}^s(\mathcal{S}_{n+R}(\bpsi)\cap \mathcal{M})=0.
    \end{equation}
\end{theorem}

Recall that the $n$- dimensional Hausdorff measure $\cH^n$ is a constant multiple of the Lebesgue measure $\mu_n$. Thus, setting $s=n$ in the above theorem establishes the convergence case of Khintchine's theorem for $\cM$. This is a weighted extension of \cite[Theorem 1.2]{BY} (also see \cite{DRN}) in the specific setting of smooth manifolds satisfying the strong condition \eqref{eq curv cond}. Further, Theorem \ref{thm hausd well app} generalizes the convergence case of \cite[Theorem 2]{beresnevich2007note} for weighted simultaneous Diophantine approximation on planar nondegenerate curves when $\psi_0\leq \psi_1$. It can be checked that for smooth planar curves, nondegeneracy (see Definition \ref{def nondeg})  is equivalent to the curvature condition \eqref{eq curv cond}. We refer to \cite[Proposition 2.13]{BY} for a proof of this fact.

As a corollary of Theorem \ref{thm hausd well app}, we can establish for the first time an upper bound for the set of weighted simultaneously well approximable points on $\cM$ corresponding to $\btau\in \left(\frac{1}{n}, \frac{1}{R}\right)^{R+1}$. In fact, since the complementary lower bound for more general manifolds has already been established in \cite[Theorem 1.1]{allen2022note} (also see \cite[Theorem~8]{ber2021lower} for an earlier result), we obtain the exact Hausdorff dimension of the set of such points.

\begin{corollary}
\label{cor lbound dim}
For $n\geq 2$ and a sufficiently small $\varepsilon_0>0$, let $\cM$ be as in \eqref{eq manifold}. Suppose condition \eqref{eq curv cond} holds. For $\btau=(\tau_0, \tau_1, \ldots, \tau_R)\in [\frac{1}{n}, \frac{1}{R})^{R+1}$ with
$\tau_0\geq \max\{\tau_1, \ldots, \tau_R\}$, we have
\begin{equation}
    \label{eq Hdim}
    \dim (\mathcal{M}\cap\mathcal{S}_{n+R}(\btau))=\frac{n+R+1+\sum_{r=1}^R(\tau_0-\tau_r)}{\tau_0+1}- R.
\end{equation}
\end{corollary}
Note that for the much more general class of $l$-nondegenerate manifolds of dimension $n$ and codimension $R$, in the case when $\tau_0=\tau_1=\ldots=\tau_R$, the currently best known result is \cite[Corollary 1.13]{DRN}, where the range of $\tau$ is given by
\begin{equation*}
    \tau<\frac{3\alpha+1}{(2(n+R)-1)\alpha+n+R}, \qquad\text{ with }\alpha:=\frac{1}{n(2l-1)(n+R+1)}.
\end{equation*}
In other words, this result covers $\tau$ close to $1/n$. Corollary \ref{cor lbound dim} is valid for a bigger range of $\tau$, but requires the manifold $\cM$ to satisfy the strong curvature assumption \eqref{eq curv cond}.

\begin{proof}[Proof of Corollary \ref{cor lbound dim}]
    Let $$s>\frac{n+R+1+\sum_{r=1}^R(\tau_0-\tau_r)}{\tau_0+1}- R=\frac{n+R+1}{\tau_0+1}- \sum_{r=1}^R\frac{(\tau_r+1)}{\tau_0+1}$$ and $\bpsi(q)=\left(q^{-\tau_0}, q^{-\tau_1}, \ldots, q^{-\tau_R}\right)$. A straightforward calculation using the lower bound on $s$ shows that \eqref{eq Hmeas gset conv} is convergent. We can thus apply Theorem \ref{thm hausd well app} to conclude that $\cH^s(\cS_{n+R}(\btau)\cap\cM)=0$, and consequently $\dim (\cS_{n+R}(\btau)\cap\cM)\le s$. Since $s>\frac{n+R+1}{\tau_0+1}- \sum_{r=1}^R\frac{(\tau_r+1)}{\tau_0+1}$ is arbitrary we conclude that 
\begin{equation}
    \label{eq pf wswa upb}
     \text{dim}(\mathcal{M}\cap\mathcal{S}_{n+R}(\btau))\leq 
     \frac{n+R+1+\sum_{r=1}^R (\tau_0-\tau_r)}{\tau_0+1}-R.
\end{equation}
To upgrade the above inequality to an equality,  we can use the lower bound provided by \cite[Theorem 1.1]{allen2022note} which is true for any $C^2$ submanifold of $\R^{M}$ of dimension $n$, whenever $\sum_{r=1}^R\tau_r<1$. Applied to our setting, it says that 
\begin{equation}
    \label{eq pf wswa lpb}
   \text{dim}(\mathcal{M}\cap\mathcal{S}_n(\tau))\geq \frac{n+R+1+\sum_{r=1}^R (\tau_0-\tau_r)}{\tau_0+1}-R.
\end{equation}
Now \eqref{eq Hdim} follows by combining \eqref{eq pf wswa upb} and \eqref{eq pf wswa lpb}. 
\end{proof}

\subsection{Novelties and comparison with previous work}
We compare our methods and results with the previous works of J.J. Huang  (for hypersurfaces) and Schindler-Yamagishi (for manifolds satisfying the condition \eqref{eq curv cond}). 

In \cite{huangduke}, Huang used a novel combination of projective duality, stationary phase and induction on scales to develop a bootstrapping argument. Starting with the trivial estimate $N_{\cM}(Q, \delta)\ll Q^{n+1}$, a repeated iteration of this process yielded the conjectured error term (of order $Q^{n-1}$) for the asymptotic expansion of $N_{\cM}(Q, \delta)$ in the case when $\cM$ is a hypersurface with non-vanishing Gaussian curvature. The bedrock of this argument was a self-improving estimate relying on the fact that the Legendre dual of a hypersurface with non-vanishing Gaussian curvature is also a hypersurface with the same property. Furthermore, the Legendre transform is an involution. Thus after every two steps of this iteration, one returns to the original counting problem 
one started with, albeit with better estimates owing to repeated applications of stationary phase and induction on scales. It was not clear, however, whether such an argument could be adapted to manifolds of arbitrary dimension, or what duality would even mean in such a setting.

In \cite{schindler2022density}, a deep insight of Schindler-Yamagishi was generalizing the notion of Legendre duality to manifolds of arbitrary dimension $n$ and codimension $R$, but satisfying the geometric condition \eqref{eq curv cond}. In essence, they exploit the curvature condition \eqref{eq curv cond} to ``freeze'' all but one codimension variables, followed by stationary phase for the family of hypersurfaces thus obtained (one for each discrete choice of the $(R-1)$ frozen variables). This paves the way for the application of van der Corput's B process for each such hypersurface, thus linking the manifold $\cM$ immersed in $\mathbb{R}^{n+R}=\mathbb{R}^{M}$ to a dual family of hypersurfaces in $\mathbb{R}^{M}$. Some major work is then involved in showing that these hypersurfaces also possess non-vanishing curvature. Once this is established, the authors use another application of stationary phase and duality (van der Corput's B process) to return to the counting function for the original manifold $\cM$. 

However, after these two steps,  
the argument proceeds by using exactly one of the codimensions to project to a lower dimensional counting problem associated to a hypersurface in $\mathbb{R}^{n+1}$ and summing trivially in the remaining in $R-1$ codimension variables. This allows for the use of the sharp estimate for the rational point count close to hypersurfaces from \cite{huangduke} as a blackbox to deduce estimates for the rational point counting function associated to this family of projected hypersurfaces. The authors thus apply a similar bootstrapping procedure as in \cite{huangduke}, but with only two steps. This is already sufficient to establish improved estimates for $\mathfrak{N}_{w, \cM}(Q, \delta)$, in a range of $\delta$ beyond Conjecture \ref{conj main}! However, the involutive nature of Legendre duality is not fully exploited. Further, the process of summing up trivially in all but one codimension variables in the second step, does not fully utilize the information in these directions.

In this paper, we establish an inductive argument in the vein of \cite{huangduke} which exploits the duality between the manifold $\cM$ and a dual family of hypersurfaces, as formulated in \cite{schindler2022density}. 
\begin{itemize}
    \item However, instead of using the main result from \cite{huangduke} for hypersurfaces as a blackbox, we use the involutive nature of the Legendre transform to return to the original counting problem associated to $\cM$ after every two steps (see Proposition \ref{prop 1}). This allows us to develop a bootstrapping argument which takes as input a trivial estimate for $\mathfrak{N}_{w, \cM}(Q, \bdel)$, and eventually yields the upper bound contained in Theorem \ref{thm main}, which is valid in a bigger range of $\delta$ than in both Conjecture \ref{conj main} and Theorem \ref{thm SY}.
    \item Moreover, in a major departure from \cite{huangduke}, the two inductive substeps (see Propositions \ref{prop dual from og} and \ref{prop og from dual}) develop a connection between counting functions associated with two entirely different geometric objects: the manifold $\cM$ of codimension $R$ immersed in $\mathbb{R}^{M}$ on one hand, and a dual family of hypersurfaces in $R^{M}$ on the other. 
    As mentioned previously, in \cite{huangduke}, both counting functions were associated to hypersurfaces with non-vanishing curvature, whereas \cite{schindler2022density} only utilized this connection in one direction. In \cite{srivastava2023density}, Technau and the author developed such a duality argument for locally flat and rough geometric objects, but they were both required to be \textit{hypersurfaces}.

    \item An important new ingredient in our argument is Proposition \ref{prop og vdc b}, which allows for passage from the sum of the dual weights (associated to hypersurfaces in $\mathbb{R}^{M}$) to the rational point count in a neighborhood of $\cM$. This should be compared to \cite[Proposition 5.3 and \S 6.3]{schindler2022density}, where the original counting problem is projected to a lower dimensional one associated to a family of hypersurfaces in $\mathbb{R}^{n+1}$, with a trivial summing up in the remaining $R-1$ directions. 
    In contrast, we utilize a dyadic version of $\mathfrak{N}_{w, \cM}(Q, \bdel)$ and retain the counting problem in all the codimensions.

    \item In order to utilize the information from all codimensions independently, our argument \textit{necessarily requires} estimates for a non-isotropic counting function (or else it is not possible to sum the dyadic terms \eqref{eq p2 1} in the proof of Proposition \ref{prop dual from og}). This inspired us to develop our argument entirely in a non-isotropic setting, leading to Theorem \ref{thm main}.

    \item The class of manifolds satisfying the condition \eqref{eq curv cond} provides a fertile ground for the above ideas to be developed. However, we expect them to be useful for the problem of counting rational points near more general manifolds like curves satisfying a weaker curvature condition. We aim to address this in a future work.

\end{itemize}

\subsection{Sharpness}
The main result from \cite{schindler2022density} (see Theorem \ref{thm SY}) broke through the $\delta>{Q^{-\frac{1}{R}}}$ threshold for manifolds $\cM$ satisfying \eqref{eq curv cond}, while our results (Corollaries \ref{cor main asump} and \ref{cor hetero}) go even beyond. A natural question is: what is the biggest range of $\delta$ in which the asymptotic/upper bound in Corollary \ref{cor main asump}/Corollary \ref{cor hetero} holds true? The answer naturally depends on the order of the error term in the asymptotic expansion \eqref{eq homog main est}, or more generally, in the upper bound in \eqref{eq hetero asymp}.

In this paper, we establish that this error term 
is of the order of $Q^{\Theta}\mathcal{E}_n(Q)$, where $\mathcal{E}_n(Q)$ is as in \eqref{eq aux error} and
$$ \Theta:=\begin{cases}
       \frac{n(n+R+1)}{n+2R}, & 1\leq R \leq 2.\\
       n+1-\frac{nR}{n+2(R-1)-\frac{4}{n}}, & R\geq 3.
    \end{cases}.$$
Proposition \ref{prop 1} states the combined effect of the two inductive substeps connecting the counting function $\mathfrak{N}_{w, \cM}(Q, \bdel)$ to its dual and vice versa (via the van der Corput B process). Neglecting logarithmic terms for the course of this discussion, a single application of Proposition \ref{prop 1} brings down the order of the error term from $Q^{\beta}$ to $Q^{\Tilde{\beta}}$, with
$$\Tilde{\beta}=n+1-\frac{nR}{n+2\left(R-\frac{n}{2\beta-n}\right)}\,.$$
Note that we have the trivial estimate $\mathfrak{N}_{w, \cM}(Q, \bdel)\ll \delp Q^{n+1}$. Starting with $\beta_0=n+1$, the above recursive relation yields a decreasing sequence $\{\beta_i\}_{i\geq 0}$ which converges to $$\frac{n(n+R+1)}{n+2R}$$ after roughly $\log \log Q$ many steps, \textit{irrespective of} whether $R\geq 2$ or $R\leq 2$. Indeed, when $R=2$, and $\cM$ satisfies \eqref{eq curv cond}, we conjecture that up to logarithmic losses, the error term is of the order of
$$Q^{\frac{n(n+R+1)}{n+2R}}.$$
When $R=1$ and $\cM$ is a hypersurface with non-vanishing curvature, this conjecture is true (as established in \cite{huangduke}); 
while our Theorem \ref{thm main} establishes that the error term is of at most this order for $R=2$.

On the other hand, for $R\geq 3$, the induction process stops once the error terms is of the order of $Q^{\Theta}$ with
\begin{equation}
    \label{eq theta crit}
    \Theta=n+1-\frac{nR}{n+2(R-1)-\frac{4}{n}}.
\end{equation}
This is because below this critical value, it is the so called ``error term'' from the stationary phase expansion for the oscillatory integrals associated to the dual hypersurfaces (the last term on the right in \eqref{eq dual vdc b}), which begins to dominate over the main terms (the first two terms on the right in \eqref{eq dual vdc b}). In other words, for larger values of $R$, the order of the error term obtained from a single term stationary phase expansion determines the critical value $\Theta$. It is interesting to ask whether it is possible to go below even this threshold when $R\geq 3$, and in particular, to realize the order of $Q^{\frac{n(n+R+1)}{n+2R}}$ for the error term.


\subsection{Outline of the paper} 
\begin{itemize}
    \item Section \ref{sec prelim} collates technical results used throughout the paper.
    \item Section \ref{sec overview} contains an outline of the proof of Theorem \ref{thm main}. 
    \begin{itemize}
        \item We first deduce it as a consequence of Proposition \ref{prop 1}. 
        \item  Next, this proposition is broken down into two further sub steps: Propositions \ref{prop dual from og} and \ref{prop og from dual}. The former uses an upper bound on the counting function associated with $\cM$ to obtain a better bound on a dual counting function associated to a family of $n$ dimensional hypersurfaces in $\mathbb{R}^{n+1}$. Proposition \ref{prop og from dual} does the reverse: it bounds the counting function associated to $\cM$ using the upper bound on the dual counting function.
        \item The proofs of Propositions \ref{prop dual from og} and \ref{prop og from dual} rely on a combination of projective duality and van der Corput's B-process to pass from $\mathcal{M}$ to the family of dual hypersurfaces, and back. This connection is made precise in Propositions \ref{prop dual vdc b} and \ref{prop og vdc b}. The former enables passage from the sum of the dual weights to rational point count in the neighborhood of $\cM$, while the latter achieves the same in the reverse direction.
    \end{itemize}
    \item Section \ref{sec dual from og} contains the proof of Proposition \ref{prop dual from og} modulo Proposition \ref{prop dual vdc b} which is itself proven in Section \ref{sec dual vdc b}. Similarly, Section \ref{sec og from dual} contains the proof of Proposition \ref{prop og from dual} modulo Proposition \ref{prop og vdc b} which is proven in Section \ref{sec og vdc b}. 
    \item Theorem \ref{thm homog main} is proven in Section \ref{sec thm homog main} using Theorem \ref{thm main}.
    \item Theorem \ref{thm hausd well app} is deduced as a corollary of Theorem \ref{thm main} in Section \ref{sec thm dio approx}.
\end{itemize}

\begin{remark}
\label{rem RH}
The condition \eqref{eq curv cond} imposes strong restrictions on how large the codimension $R$ of $\cM$ can be with respect to its dimension $n$. Indeed, as mentioned in \cite{schindler2022density}, the problem of finding $n$ dimensional matrices satisfying \eqref{eq curv cond} is connected to question of determining the number of linearly independent vector fields on spheres. An $R$-tuple of symmetric matrices satisfying the condition gives rise to a system of $(R-1)$ linearly
independent vector fields on the $n-1$ dimensional sphere in $\mathbb{R}^n$. We refer the interested reader to \cite{adams1962vector} for further reading. In particular, 
\begin{equation}
    \label{R RH}
    R\leq \RH(n),
\end{equation}
where $\RH(n)$ are Radon-Hurwitz numbers (see \cites{hurwitz, radon}) defined as follows:  if  $n =(2n_1+1) 2^{4n_2+n_3}$ with $n_3\in \{0,1,2,3\}$ and for some $n_1\in  \{0,1,2,3,\dots\}$, then  $\RH(n) =8n_2+2^{n_3}$.
Notice that for odd $n$, we have 
\begin{equation}
    \label{odd RH}
    R\leq \RH(n)=1.
\end{equation}
For the interested reader, we also point out that the works of Radon and Hurwitz had originally been in the setting of Heisenberg type groups, where an analogous connection dictates how large the dimension of the centre of the Lie algebra can be as compared to the dimension of the group. We refer to \cite{Kaplan1980} and \cite{roos2022lebesgue} for further reading.
\end{remark}

\subsection{Notation} 
All vectors shall be denoted by boldface letters, e.g., $\bx, \by, \bz, \bdel, \bk$. We shall use $|\cdot|$ to denote the $\ell^\infty$ norm of the vector under consideration. In other words, for $\bz\in \mathbb{R}^k$ (where $k\in \mathbb{N}$), 
$$|\bz|:=\max_{1\leq i\leq k} |z_i|.$$ Given $\varepsilon>0$ and $\bx\in \mathbb{R}^n$, $B_{\varepsilon}(\bx)$ shall denote the open ball defined with respect to the metric induced by this norm, centred at $\bx$ and of radius $\varepsilon$; i.e.,
$$B_{\varepsilon}(\bx):=\{\by: |\bx-\bz|< \varepsilon\}.$$

For an open set $\mathscr{X}\in \mathbb{R}^n$ and $\ell\in \mathbb{N}$, we shall denote the set of all $\ell$-times continuously differentiable functions defined on $\mathscr{X}$ by $C^{\ell}(\mathscr{X})$, the set of all smooth functions defined on $\mathscr{X}$ by $C^{\infty}(\mathscr{X})$ and the set of all smooth functions defined on $\mathscr{X}$ with compact support by $C^{\infty}_0(\mathscr{X})$. The gradient of a function $f\in C^{1}(\mathscr{X})$ shall be denoted by
$$\nabla f:=\left(\frac{\partial f}{\partial x_1}, \ldots, \frac{\partial f}{\partial x_n}\right).$$
For $t\in \mathbb{R}$, we define
$$e(t):=e^{2\pi i t}.$$

In the following, given a positive integer $k$ and functions $A, B: \mathscr{X}\subset\mathbb{R}^k\to \mathbb{C}$, we shall use the notation $A\ll B$ to denote the fact that $|A(\bx)|<C|B(\bx)|$ for all $\bx\in \mathscr{X}$, where $C>0$ is a constant which is allowed to depend on $w$ and $\cM$, and therefore also on 
\begin{itemize}
    \item the dimension $n$ and codimension $R$,
    \item $\varepsilon_0$ and $\bx_0$, 
    \item the constant $\mathfrak{C}_0$ in condition \eqref{eq curv est},
    \item upper bounds for the functions $f_r$ ($1\leq r\leq R$) and $w$, and for their finitely many derivatives, on the domain ${B_{4\varepsilon_0}(\bx_0)}$.
\end{itemize}

\subsection{Acknowledgement} The author is
supported by the Deutsche Forschungsgemeinschaft (DFG, German Research Foundation) under Germany’s Excellence Strategy 
- EXC-2047/1 - 390685813 as well as SFB 1060. She is grateful to Lillian Pierce, Andreas Seeger and Niclas Technau for insightful discussions at different stages of this project. The author also thanks Lars Becker for carefully reading Section \ref{sec introduction} and providing useful comments, and Jing-Jing Huang for directing her to a precise version of Conjecture \ref{conj main}. Finally, she is thankful to the anonymous referee for a thorough review of this manuscript and for helpful suggestions that improved the exposition. 

\section{Preliminaries}
\label{sec prelim}
In this section, we collect various technical results used throughout the paper.

\subsection{Oscillatory Integral Estimates}

\begin{lemma}[Non-Stationary Phase]
\label{lem non st phase}   
Let $d, K\in \mathbb{Z}_{>0}$ and $\mathscr{U}\subseteq \mathbb{R}^d$ be a bounded open set. Let $\omega\in C^{\infty}(\mathbb{R}^d)$, with $\textrm{supp }(\omega)\subseteq \mathscr{U}$ and $\varphi\in C^{\infty}(\mathscr{U})$ with $\nabla \varphi (\bx)\neq \bzero$ for all $\bx\in \textrm{supp }(\omega)$. Then for any $\lambda>0$, we have
$$\left|\int_{\mathbb{R}^d}\omega(\bx) e\left(\lambda \varphi(\bx)\right)\, d\bx\right|\ll \lambda^{-K+1},$$
where the implied constant depends only on $K, d$, upper bounds for the absolute values of finitely many derivatives of $\omega$ and $\varphi$ on $\mathscr{U}$, and the lower bound for $|\nabla \varphi|$ on $\textrm{supp }\omega$.
\end{lemma}
\begin{proof}
    See \cite[Theorem 7.7.1]{hormander}
\end{proof}

\begin{lemma}[Stationary Phase]
\label{lem st phase}    
Let $\mathscr{U}, \mathscr{U}_1\subseteq \mathbb{R}^d$ be bounded open sets, with $\overline{\mathscr{U}}\subset\mathscr{U}_1$. Let $\omega\in C^{\infty}(\mathbb{R}^d)$, with $\textrm{supp }(\omega)\subseteq \mathscr{U}$ and $\varphi\in C^{\infty}(\mathscr{U}_1)$. Suppose there exists $\bv_0\in \mathscr{U}$ with $\nabla \varphi(\bv_0)=0$, and $\nabla \varphi(\bx)\neq0$ for all $\bx\in \overline{\mathscr{U}}\setminus\{\bv_0\}$. Further, let
$$\Delta:=\left|\det\, H_{\varphi}(\bv_0)\right|\neq 0$$ and $\sigma$ be the signature of $H_{\varphi}(\bv_0)$. Then for any $\lambda>0$, we have
\begin{equation}
    \label{eq st phase}
    \int_{\mathbb{R}^d} e\left(\lambda\varphi(\bx)\right) \omega(\bx)\, d\bx= \lambda^{-\frac{d}{2}}\Delta^{-\frac{1}{2}}e\left(\lambda\varphi(\bv_0)+\frac{\sigma}{8}\right)\left(\omega(\bv_0)+O\left(\lambda^{-1}\right)\right),
\end{equation}
where the implicit constant depends only on $d$, upper bounds for (the absolute values of) finitely many derivatives of $\omega$ and $\varphi$ on $\mathscr{U}_1$, an upper bound for 
$$\sup_{\bx\in \mathscr{U}_1\setminus\{\bv_0\}}\frac{|\bx-\bv_0|}{\nabla \varphi (\bx)}$$ and a lower bound for $\Delta$.
\end{lemma}

\begin{proof}
    See \cite[Chapter VIII, Proposition 6]{stein1993harmonic} or \cite[Theorem 7.7.5]{hormander}.
\end{proof}

\subsection{Selberg Magic Functions}
\label{subsec selb magic func}
Let $I=(\alpha, \beta)$ be an arc of $\mathbb{R}/\mathbb{Z}$ with $0<\beta-\alpha<1$, and let $\mathds{1}_{I}$ denote its characteristic function. Then given $J\in\mathbb{Z}_{>0}$, there exist finite trigonometric polynomials $S_{J^{\pm}}: [0, 1]\to \mathbb{R}$ of degree at most $J$
$$S_{J^{\pm}}(x):=\sum_{j=-J}^{J}\Hat{S}_{J^{\pm}}(j)e(jx), $$
satisfying the following properties:
\begin{align}
    \label{eq selb1}
    S_{J^-}(\theta)&\leq \mathds{1}_I(\theta)\leq S_{J^+}(\theta),\qquad \theta\in \mathbb{R}/\mathbb{Z};\\
    \label{eq selb0}
    \Hat{S}_{J^{\pm}}(0)&=\beta-\alpha\pm \frac{1}{J+1};\qquad \text{ and}\\
    \label{eq selbj}
    |\Hat{S}_{J^{\pm}}(j)|&\leq \frac{1}{J+1}+\min\left(\beta-\alpha, \frac{1}{\pi|j|}\right).
\end{align}
We refer to \cite[Chapter 1]{montgomery1994ten} for details about the construction of these functions.

\subsection{The Fej\'er Kernel}
For $D\in \mathbb{Z}_{>0}$, let $\mathcal{F}_D: [0, 1]\to \mathbb{R}$ be the Fej\'er kernel of degree $D$ given by
\begin{equation}
    \label{def fejer ker}
    \mathcal{F}_D(\theta)= \sum_{d=-D}^{D}\frac{D-|d|}{D^2} e(d\theta)=\left(\frac{\sin(\pi D\theta)}{D\sin(\pi\theta)}\right)^2.
\end{equation}
Let $\delta^*\in (0, 1/2)$ be such that $D=\left\lfloor \frac{1}{2\delta^*}\right\rfloor.$ Since $|\sin (\pi x)|\geq 2x$ for $x\in [0, 1/2]$, we have
$$\left(\frac{\sin(\pi D\theta)}{D\sin(\pi\theta)}\right)^2\geq \left(\frac{2 D\|\theta\|}{D\pi \|\theta\|}\right)^2=\frac{4}{\pi^2}$$
whenever $\|\theta\|\in \left(0, \delta^*\right)$.
In other words, letting $\mathds{1}_{\delta^*}$ denote the characteristic function of the set $\{x\in \mathbb{R}: \|x\|\leq \delta^*\}$, we have
\begin{equation}
    \label{eq fejer char est}
    \mathds{1}_{\delta^*}(\theta)\leq \frac{4}{\pi^2} \mathcal{F}_D(\theta).
\end{equation}

\section{Overview of the Proof of Theorem \ref{thm main}}
\label{sec overview}
We now fix $n\geq 2$ and a sufficiently small $\varepsilon_0>0$. Further, define \begin{equation}
\label{def beta stop}
    \beta_{\textrm{st}}:= \max\left(\frac{n(n+R+1)}{n+2R}, \frac{n(n+1)}{n+2}\right).
\end{equation}
Observe that $\beta_{\textrm{st}}=\Theta$ for $R\in \{1, 2\}$,
while $ \beta_{\textrm{st}}>\Theta$ for $ R\geq 3$.

Theorem \ref{thm main} will be a consequence of the following self-improving estimate for the manifold $\mathcal{M}$, assumed to satisfy condition \eqref{eq curv cond}.

\begin{proposition}
\label{prop 1}
Suppose there exists $\beta\in \left[\beta_{\textrm{st}}\,, n+1\right]$, $A\geq 1$ and $a\geq 0$ such that for all $Q\geq 1$ and for all $\bdel\in (0, 1/2)^R$, it is true that
\begin{equation}
    \label{eq p1 hypo}
    \mathfrak{N}_{w, \cM}(Q, \bdel)\leq A\left(\delp Q^{n+1}+ \sum_{r=1}^R \delpr Q^{\frac{(n+1)(R-1)}{R}} Q^{\frac{\beta}{R}} +Q^{\beta}\left(\log 4Q\right)^{\aR}\right).
\end{equation}

Then there exists a positive constant $C_0$ depending only on $w$ and $\cM$, such that for all $Q\geq 1$ and for all $\bdel\in (0, 1/2)^R$, we also have 
\begin{equation}
    \label{eq p1 concl}
    \mathfrak{N}_{w, \cM}(Q, \bdel)\leq C_0\,\delp Q^{n+1} +C_0 A 
   \left( \sum_{r=1}^R \delpr Q^{\frac{(n+1)(R-1)}{R}} Q^{\frac{\tilde{\beta}}{R}} +Q^{\tilde{\beta}}\left(\log 4Q\right)^{\aRR}\right);
\end{equation}
with 
\begin{equation}
    \label{eq beta rel}
    \tilde{\beta}=n+1-\frac{n R}{n+ 2\left(R-\frac{n}{2\beta-n}\right)}.
\end{equation}
\end{proposition}
\begin{proof}[Proof of Theorem \ref{thm main}, assuming Proposition \ref{prop 1}]
Using the trivial estimate    
\begin{equation*}
     \mathfrak{N}_{w, \cM}(Q, \bdel)\leq \delp Q^{n+1} +Q^{n+1}\leq \delp Q^{n+1}+ \sum_{r=1}^R \delpr Q^{\frac{(n+1)(R-1)}{R}} Q^{\frac{n+1}{R}} +Q^{n+1}
\end{equation*}
as the starting point, we apply Proposition \ref{prop 1} repeatedly. 
We first consider the case when $\beta_{\textrm{st}}=\frac{n(n+R+1)}{n+2R}$, which arises when $R\in \{1, 2\}$.

Applying Proposition \ref{prop 1} $N$ many times, we get
\begin{equation}
    \label{eq betaN}
     \mathfrak{N}_{w, \cM}(Q, \bdel)\leq C_0\delp Q^{n+1} +C^N \left( \sum_{r=1}^R \delpr Q^{\frac{(n+1)(R-1)}{R}} Q^{\frac{\beta_N}{R}} +Q^{\beta_N}\left(\log 4Q\right)^{N(R-1)+R}\right),
\end{equation}
where $C_0, C\geq 1$ are constants depending on $w$ and $\cM$; and $\beta_N$ is obtained from the recursive relation
\begin{equation}
    \label{eq beta rec rel}
    {\beta}_i=n+1-\frac{n R}{2 R+ n\left(1-\frac{2}{2\beta_{i-1}-n}\right)},\qquad\beta_0:=n+1 .
\end{equation}

To determine how rapidly this sequence converges to $\frac{n(n+R+1)}{n+2R}$, we calculate
\begin{align}
    \beta_i-\frac{n(n+R+1)}{n+2R}&= n+1-\frac{n R}{2 R+ n\left(1-\frac{2}{2\beta_{i-1}-n}\right)}-\frac{n(n+R+1)}{n+2R}\nonumber\\
    &=\frac{(n+2)R}{n+2R}-\frac{n R}{2 R+ n\left(1-\frac{2}{2\beta_{i-1}-n}\right)}.
    \label{eq best1}
\end{align}
We can rewrite the last expression above as
\begin{align}
&\frac{(n+2)R}{n+2R}-\frac{n R(2\beta_{i-1}-n)}{(n+2R)(2\beta_{i-1}-n)-2n}\nonumber\\
&= \frac{R}{(n+2R)(2\beta_{i-1}-n)-2n}\left[\frac{(n+2)((n+2R)(2\beta_{i-1}-n)-2n)-n(n+2R)(2\beta_{i-1}-n)}{n+2R}\right].
\label{eq best2}
\end{align}
The expression in square brackets equals
\begin{align}
\label{eq best3}
    \frac{-2n^2+2((n+2R)(2\beta_{i-1}-n)-2n)}{n+2R}&=4\beta_{i-1}-\frac{2n^2+2n(n+2R)+4n}{n+2R}\nonumber\\
    &=4\beta_{i-1}-\frac{4n(n+R+1)}{n+2R}.
\end{align}
From \eqref{eq best1}, \eqref{eq best2} and \eqref{eq best3}, we conclude that
\begin{equation}
 \beta_i-\frac{n(n+R+1)}{n+2R}=\frac{4R}{(2\beta_{i-1}-n)(n+2R)-2n}\left(\beta_{i-1}-\frac{n(n+R+1)}{n+2R}\right). \label{eq beta est1}
\end{equation}
Now since $\beta_{i-1}\geq \frac{n(n+R+1)}{n+2R}$, we have
$$(2\beta_{i-1}-n)(n+2R)-2n\geq n(n+2)-2n=n^2.$$
Plugging this into \eqref{eq beta est1}, we get
\begin{equation}
\label{eq beta est2.1}
\beta_i-\frac{n(n+R+1)}{n+2R}\leq \frac{4R}{n^2}\left(\beta_{i-1}-\frac{n(n+R+1)}{n+2R}\right).
\end{equation}
We start with the estimate
$$\beta_0-\frac{n(n+R+1)}{n+2R}=n+1-\frac{n(n+R+1)}{n+2R}=\frac{(n+2)R}{n+2R}\leq R.$$
After $N$ steps, using \eqref{eq beta est2.1}, we get
\begin{equation}
\label{eq beta est2}
\beta_N-\frac{n(n+R+1)}{n+2R}\leq \left(\frac{4R}{n^2}\right)^N R. 
\end{equation}

We now consider two subcases.
When $n\geq 3$, it follows from Remark \ref{rem RH} and the definition of the Radon-Hurwitz numbers that $R<n$ . Further, recall \eqref{odd RH} which says that $R=1$ whenever $n$ is odd. 
Thus,
\begin{equation}
\label{eq Rn ratio}
\frac{4R}{n^2}\leq \left\{\begin{array}{lr}
        \frac{4}{5} , & \text{if } n\geq 5\\\\
        \frac{3}{4} , & \text{if } n=4\\\\
        \frac{4}{9} , & \text{if } n=3.
        \end{array}\right.    
\end{equation}
Combining the above with \eqref{eq beta est2}, we get
$$\beta_N-\frac{n(n+R+1)}{n+2R}\leq \left(\frac{4}{5}\right)^N R.$$
Thus, after $$N=\left\lfloor\frac{\log R+\log\log 4Q}{\log (5/4)}\right\rfloor$$ many steps, we conclude from $\eqref{eq betaN}$ that 
$$\mathfrak{N}_{w, \cM}(Q, \bdel)\ll \delp Q^{n+1} +\sum_{r=1}^R \delpr Q^{\frac{(n+1)(R-1)}{R}} Q^{\frac{n(n+R+1)}{R(n+2R)}}(\log 4Q)^{\mathfrak{c}_2}  +Q^{\frac{n(n+R+1)}{n+2R}}\exp\left({\mathfrak{c}_2}(\log\log 4Q)^2\right),$$
for $R=2$; and
$$\mathfrak{N}_{w, \cM}(Q, \bdel)\ll \delp Q^{n+1} +\sum_{r=1}^R \delpr Q^{\frac{(n+1)(R-1)}{R}} Q^{\frac{n(n+R+1)}{R(n+2R)}}(\log 4Q)^{\mathfrak{c}_2}  +Q^{\frac{n(n+R+1)}{n+2R}}(\log 4Q)^{\mathfrak{c}_2},$$
for $R=1$. Here $\mathfrak{c}_2>0$ and the implied constant depend only on $w$ and $\cM$. This finishes the proof for $n\geq 3$ and $R\in \{1, 2\}$.

In the subcase when $n=2$ and $R=1$, the relation \eqref{eq beta est1} reduces to 
$$\beta_i-2=\frac{\beta_{i-1}-2}{2\beta_{i-1}-3}.$$
or equivalently, the identity
$$\frac{1}{\beta_i-2}=2+\frac{1}{\beta_{i-1}-2},$$
with $\beta_0=3$.
This lets us conclude that
$$\beta_i=2+\frac{1}{2i+1}.$$
This sequence also converges to the desired value of $2$, albeit at a much slower rate. Indded, after
$$N=\lfloor\sqrt{\log 4Q}\rfloor$$ many steps, we obtain using $\eqref{eq betaN}$ that
$$\mathfrak{N}_{w, \cM}(Q, \delta_1)\ll \delta_1 Q^{3} +Q^2\exp(\mathfrak{c}_1\sqrt{\log 4Q}),$$
for a large enough constant $\mathfrak{c}_1>0$ depending only on $w$ and $\cM$. This establishes our result also in the case when $n=2$.

We now come to the case when $\beta_{\textrm{st}}=\frac{n(n+1)}{2}$, which arises when $R\geq 3$. The argument proceeds in almost the same way as the previous case, except for one key difference. The sequence $\{\beta_i\}$ is still defined by the recursive relation \eqref{eq beta rec rel} and therefore converges to $\frac{n(n+R+1)}{n+2R}$. However, Proposition \ref{prop 1} can be applied only until $$\beta_i\geq \frac{n(n+1)}{2}>\frac{n(n+R+1)}{n+2R}.$$
Consequently, let $\beta_{k-1}, \beta_{k}$ be such that
$$\beta_{k}< \frac{n(n+1)}{2}\leq \beta_{k-1}.$$
Applying Proposition $\ref{prop 1}$ $k$ many times, we obtain
\begin{align*}
    \mathfrak{N}_{w, \cM}(Q, \bdel)&\leq C_0\,\delp Q^{n+1} +C^{k} \left( \sum_{r=1}^R \delpr Q^{\frac{(n+1)(R-1)}{R}} Q^{\frac{{\beta_k}}{R}} +Q^{{\beta_k}}\left(\log 4Q\right)^{\aRR}\right),\\
   &\leq C_0\,\delp Q^{n+1} +C^{k}\left( \sum_{r=1}^R \delpr Q^{\frac{(n+1)(R-1)}{R}} Q^{\frac{{n(n+1)}}{2R}} +Q^{\frac{n(n+1)}{2}}\left(\log 4Q\right)^{\aRR}\right),
\end{align*}
where $C_0, C_1$ are positive constants depending only on $w$ and $\cM$. 
We apply Proposition \ref{prop 1} one more time, with $\beta=\frac{n(n+1)}{2}$, which gives
\begin{align}
    \label{eq theta estimate}
    \mathfrak{N}_{w, \cM}(Q, \bdel)\leq C_0\,\delp Q^{n+1} +C^{k+1} \left( \sum_{r=1}^R \delpr Q^{\frac{(n+1)(R-1)}{R}} Q^{\frac{{\Theta}}{R}} +Q^{{\Theta}}\left(\log 4Q\right)^{\aRR}\right),
\end{align}
with
$$\Theta=n+1-\frac{n R}{n+ 2\left(R-\frac{n}{\frac{2n(n+1)}{2}-n}\right)}=n+1-\frac{nR}{n+2(R-1)-\frac{4}{n}}\,.$$
Since $\Theta\geq \frac{n(n+R+1)}{n+2R}$, we conclude using \eqref{eq beta est2} and \eqref{eq Rn ratio} that
$$\beta_k-\Theta\leq \beta_k-\frac{n(n+R+1)}{n+2R}\leq \left(\frac{4}{5}\right)^{k} R.$$
Therefore we can again bound the number of steps required by
$$\left\lfloor\frac{\log R+\log\log 4Q}{\log (5/4)}\right\rfloor +1. $$
Plugging the above into \eqref{eq theta estimate}, we get
$$\mathfrak{N}_{w, \cM}(Q, \bdel)\ll \delp Q^{n+1} +\sum_{r=1}^R \delpr Q^{\frac{(n+1)(R-1)}{R}} Q^{\frac{\Theta}{R}}(\log 4Q)^{\mathfrak{c}_2}  +Q^{\Theta}\exp\left({\mathfrak{c}_2}(\log\log 4Q)^2\right),$$
for $R>1$; and
$$\mathfrak{N}_{w, \cM}(Q, \bdel)\ll \delp Q^{n+1} +\sum_{r=1}^R \delpr Q^{\frac{(n+1)(R-1)}{R}} Q^{\frac{n(n+R+1)}{R(n+2R)}}(\log 4Q)^{\mathfrak{c}_2}  +Q^{\Theta}(\log 4Q)^{\mathfrak{c}_2},$$
for $R=1$. Here $\mathfrak{c}_2>0$ and the implied constant depend only on $w$ and $\cM$.

This establishes \eqref{eq main est} also in the case when $R\geq 3$ (note that the condition \eqref{R RH} forces $n$ to be at least $3$), and finishes the proof.
\end{proof}

Proposition \ref{prop 1} will be a direct consequence of two inductive sub steps, built on the connection between the rational point count near the manifold $\cM$ and the sum of certain ``dual'' weights associated to a \emph{family} of compact hypersurfaces in $\mathbb{R}^{n+1}$. We need some technical preparation first.

Let $\mathds{1}_{\delta}$ denote the characteristic function of the set $\{\theta: \|\theta\|\leq \delta\}$, and for $1\leq r\leq R$, set 
\begin{equation}
    \label{eq Jdef}
    J_r:=\left\lfloor \frac{1}{2\delta_r}\right\rfloor.
\end{equation}

Using the Selberg magic functions of degree $J_r$ to estimate $\mathds{1}_{\delta_r}$ as in \eqref{eq selb1},  we can bound 
\begin{equation*}
    \mathfrak{N}_{w, \cM}(Q, \bdel)=\sum_{\substack{\ba\in \mathbb{Z}^n\\1\leq q\leq Q}} w\left(\frac{\ba}{q}\right)\prod_{r=1}^R\mathds{1}_{\delta_r}\left(qf_r\left(\frac{\ba}{q}\right)\right)\leq \sum_{\substack{\ba\in \mathbb{Z}^n\\1\leq q\leq Q}} w\left(\frac{\ba}{q}\right)\prod_{r=1}^R S_{J^+_r}\left(qf_r\left(\frac{\ba}{q}\right)\right).
\end{equation*}
Expanding the Selberg magic functions into their fourier series and multiplying, we get
\begin{equation*}
\prod_{r=1}^R S_{J^+_r}\left(qf_r\left(\frac{\ba}{q}\right)\right)=\sum_{\substack{\bj\in \mathbb{Z}^R:\\|j_r|\leq J_r}} \left(\prod_{r=1}^R \widehat{S_{J^+_r}}(j_r)\right)\exp \left(\sum_{r=1}^Rj_r q f_r \aq\right).
\end{equation*}
The upshot is that
\begin{equation}
    \label{eq exp sum 1}
    \Nwm \leq \sum_{\substack{\bj\in \mathbb{Z}^R:\\|j_r|\leq J_r\\1\leq r\leq R}} \left(\prod_{r=1}^R \widehat{S_{J^+_r}}(j_r)\right)\sum_{\substack{\ba\in \mathbb{Z}^n\\1\leq q\leq Q}} w\left(\frac{\ba}{q}\right)\exp \left(\sum_{r=1}^Rj_r q f_r \aq\right).
\end{equation}
Using \eqref{eq selb0} with $\beta-\alpha=2\delta_r$ and $J=J_r$ for each $1\leq r\leq R$, we conclude that the contribution from the regime when $\bj=\bzero$ is
\begin{equation}
    \label{eq zerofreq}
    \prod_{r=1}^R\left(2\delta_r+\frac{1}{J_r+1}\right)\sum_{\substack{\ba\in \mathbb{Z}^n\\1\leq q\leq Q}} w\left(\frac{\ba}{q}\right)\ll_w \delp Q^{n+1}.
\end{equation}
Combining \eqref{eq exp sum 1}, \eqref{eq zerofreq} and the upper bound 
$$|\widehat{S_{J_r^+}}(j_r)|\leq 
\frac{1}{J_r+1}+\min \left(2\delta_r, \frac{1}{\pi|j_r|}\right)
\leq \frac{1}{|j_r|+1}
,$$
we get
\begin{equation}
    \label{eq exp sum 2}
    \Nwm \leq C_{w} \delp Q^{n+1}+ \sum_{\substack{\bj\in \mathbb{Z}^R\setminus\{\bzero\}:\\|j_r|\leq J_r\\1\leq r\leq R}} \left(\prod_{r=1}^R \frac{1}{|j_r|+1}\right)\left|\sum_{\substack{\ba\in \mathbb{Z}^n\\1\leq q\leq Q}} w\left(\frac{\ba}{q}\right)\exp \left(\sum_{r=1}^Rj_r q f_r \aq\right)\right|.
\end{equation}
For $1\leq s\leq R$ and $X\in \mathbb{Z}_{>0}$, we define the ``pencil set'' of indices corresponding to the  codimension indexed by $s$ to be
\begin{equation}
    \label{eq J index set}
    \mathcal{J}^s(X):=\{\bj\in\mathbb{Z}_{\geq 0}^{R}: 0<\|\bj\|_{\infty}=j_s\leq X\}.
\end{equation}
For later use, we also introduce the slightly larger set
\begin{equation}
    \label{eq Jbig index set}
    \mathcal{J}^s_{b}(X):=\{\bj\in\mathbb{Z}_{\geq 0}^{R}:  0<\|\bj\|_{\infty}\leq 2j_s\leq 2X\}.
\end{equation}
Further, for  each $1\leq s\leq R$, $\bj\in \mathcal{J}^s(X)$ and $\bgam\in \{-1, 1\}^{R}$, let 
\begin{equation}
    \label{eq def Fj}
    F_{s, \bj, \bgam}:=\gamma_sf_s+\sum_{\substack{1\leq r\leq R\\ r\neq s}} \gamma_r\frac{j_r}{j_s} f_r,
\end{equation}
and
\begin{align}
    {\mathfrak{N}}^{s, \bgam}_{w, \cM}(Q, \bdel)&:= \delp Q^{n+1}+ 
    \sum_{\bj\in \mathcal{J}^s(J_s)}\left(\prod_{r=1}^R \frac{1}{j_r+1}\right)\left|\sum_{\substack{\ba\in \mathbb{Z}^n\\1\leq q\leq Q}} w\left(\frac{\ba}{q}\right)\exp \left(\sum_{r=1}^R\gamma_rj_r q f_r \aq\right)\right|\nonumber\\
    &= \delp Q^{n+1}+ \sum_{\bj\in \mathcal{J}^s(J_s)}\left(\prod_{r=1}^R \frac{1}{j_r+1}\right)\left|\sum_{\substack{\ba\in \mathbb{Z}^n\\1\leq q\leq Q}} w\left(\frac{\ba}{q}\right)\exp \left(qj_s F_{s,\bj, \bgam}\aq\right)\right|.
    \label{eq Nog 11 sum}
\end{align}
Then it follows from the triangle inequality that
\begin{equation}
    \label{eq split codim}
    {\mathfrak{N}}_{w, \cM}(Q, \bdel)\ll \sum_{\bgam\in \{-1, 1\}^R}\sum_{s=1}^r {\mathfrak{N}}^{s, \bgam}_{w, \cM}(Q, \bdel).
\end{equation}
Our argument shall be independent of the signs of the coefficients $j_r$. Thus by conjugation, if need be, we can always reduce matters to the case when $\bgam=(1, 1, \ldots, 1)$. Henceforth, we shall specialize to this choice of $\bgam$ and suppress it from notation. In principle, the same argument can be made for the choice of a distinguished codimension $s$, but for clarity, we shall state the subsequent propositions for a general $s\in \{1, \ldots, R\}$.

Let \begin{equation}
    \label{eq def og domains}
    \mathscr{D}:= B_{2\varepsilon_0}(\bx_0),\qquad U:=\supp\, w\subseteq B_{\varepsilon_0}(\bx_0),
\end{equation}
with $\varepsilon_0$ chosen small enough so that there exists a constant $\mathfrak{C}_0>1$ such that for all $1\leq s\leq R$ and $\mathbf{t}\in [-2, 2]^{R-1}$, we have
\begin{equation}
\label{eq curv est}
\frac{1}{\mathfrak{C}_0}\leq \left|\det\, H_{f_s+\sum_{\substack{r=1\\ r\neq s}}^Rt_rf_r}\right|\leq \mathfrak{C}_0;
\end{equation}
and the maps
$$\bx \to \nabla_{\bx}\left(f_s+ \sum_{\substack{1\leq r\leq R\\ r\neq s}}t_r f_r\right)$$ are smooth diffeomorphisms on $\mathscr{D}$. That such a choice of $\varepsilon_0$ exists can be seen by using the inverse function theorem and compactness arguments (see, for instance, \cite[Lemma 3.4]{schindler2022density} for the details). 

Then $\nabla F_{s, \bj}$ is a diffeomorphism on $\mathscr{D}$ and $U$ for all $s\in \{1, \ldots, R\}$ and $\bj\in \mathcal{J}_b^s(J_s)$, and therefore also for all $\bj\in \mathcal{J}^s(J_s)$. Let
\begin{equation}
    \label{eq def dual domains}
    \mathscr{R}_{s, \bj}:= \nabla F_{s, \bj} (\mathscr{D}),\qquad V_{s, \bj}:=\nabla F_{s, \bj} (U).
\end{equation}
Note that there exists a compact set $\mathscr{L}$, independent of $s$ and $J_1, \ldots, J_R$, such that $V_{s, \bj}\subseteq \mathscr{R}_{s, \bj}\subseteq \mathscr{L}$ for all $\bj\in \mathcal{J}_b^s(J_s)$. For these $\bj$ tuples, we can now define the (Legendre) dual family of functions $F_{s, \bj}^*:\mathscr{R}_{s, \bj}\to \mathbb{R}$ by
\begin{equation}
    \label{eq def dual Fj}
    F_{s,\bj}^*(\by):=\by\cdot(\nabla F_{s,\bj})^{-1}(\by)-\left(F_{s,\bj}\circ (\nabla F_{s,\bj})^{-1}\right)(\by).
\end{equation}
For each dual function $F_{s,\bj}^*$, its gradient $\nabla F_{s,\bj}^*$ is a smooth diffeomorphism mapping $\mathscr{R}_{s,\bj}$ onto $\mathscr{D}$. For $\by=\nabla F_{s,\bj}(\bx)$, we have
\begin{equation}
    \label{eq inv deriv}
    \nabla F_{s, \bj}^*(\by)=\bx,\qquad H_{F_{s, \bj}^*}(\by)=H_{F_{s, \bj}}(\bx)^{-1}.
\end{equation}
It can also be verified that the Legendre transform is an involution, i.e. $$\left(F_{s,\bj}\right)^{**}=F_{s,\bj}.$$
For $s\in \{1, \ldots, R\}$ and $\bj\in \mathcal{J}^b_s(J_s)$, we define the dual family of weights $w_{s, \bj}^*:{V}_{s,\bj}\to \mathbb{R}$ given by 
\begin{equation}
    \label{eq def dual wj}
    w_{s, \bj}^*(\by):=w\circ (\nabla F_{s,\bj})^{-1}.
\end{equation}
Finally, for $Q^*\geq 1$, $\delta^*\in (0, 1/2)$ and $s\in \{1, \ldots, R\}$, define
\begin{equation}
    \label{eq def dual weight Hessian}
    {\mathfrak{N}}^{*, s}_{w, \cM}(Q^*, \delta^*):=
    \sum_{\bj\in \mathcal{J}^s(Q^*)}\sum_{\substack{\ba\in \mathbb{Z}^n\\\|j_s F^*_{s,\bj}(\ba/j_s)\|<\delta^{*}}}\frac{w_{s,\bj}^*\ajs}{\sqrt{\left|\det\, H_{F_{s,\bj}}\left((\nabla F_{s,\bj})^{-1}\ajs\right)\right|}}
\end{equation}

We are now ready to state the first sub step in the proof of Proposition \ref{prop 1}, which converts an upper bound for ${\mathfrak{N}}^s_{w, \cM}(Q, \delta)$ into an improved upper bound for the sum of the dual family of weights ${\mathfrak{N}}^{*, s}_{w, \cM}(Q^*, \delta^*)$. In analogy with \eqref{def beta stop}, we set
\begin{equation}
    \label{def alpha stop}
    \alpha_{\textrm{st}}:= \max\left(\frac{n(n+R+1)}{n+2}, n+R-1-\frac{2}{n}\right).
\end{equation}   

\begin{proposition}
    \label{prop dual from og}
    Suppose there exist $\beta\in \left[\beta_{\textrm{st}}\,, n+1\right]$, $A_1\geq 1$ and $a_1\geq 0$ such that for all $Q\geq 1$ and for all $\bdel\in (0, 1/2)^R$, it is true that
\begin{equation}
    \label{eq p2 hypo}
   \mathfrak{N}_{w, \cM}(Q, \bdel)\leq
   A_1\left(\delp Q^{n+1}+ \sum_{r=1}^R \delpr Q^{\frac{(n+1)(R-1)}{R}} Q^{\frac{\beta}{R}} +Q^{\beta}\left(\log 4Q\right)^{\aone}\right).
\end{equation}
Then there exists a positive constant $C_1$ depending only on $w$ and $\cM$ such that for all $s\in \{1, \ldots, R\}$, $Q^*\geq 1$ and $\delta^*\in (0, 1/2)$, we have 
\begin{equation}
    \label{eq p2 concl}
    \DNwms \leq C_1\delta^* (Q^*)^{n+R} +C_1A_1\left(Q^*\right)^{\alpha}\left(\left(\log 4Q^*\right)\left(\frac{n+2R}{n+2}\right)\right)^{\aone},
\end{equation}
with 
\begin{equation}
    \label{eq alpha beta rel}
    \alpha =\max\left(n+R-\frac{n }{2 \beta- n}, n+R-1-\frac{2}{n}\right)
    \in \left[\alpha_{\textrm{st}}\,, n+R\right].
\end{equation}
\end{proposition}

The second sub step works in the reverse direction: it converts the upper bound for the sum of the above dual family of weights into an improved upper bound for ${\mathfrak{N}}^s_{w, \cM}(Q, \delta)$.
\begin{proposition}
    \label{prop og from dual}
    Suppose there exists $\alpha\in \left[\frac{n(n+R+1)}{n+2}, n+R\right]$, $A_2\geq 1$ and $a_2\geq 0$ such that for all $s\in \{1, \ldots, R\}$, $Q^*\geq 1$ and $\delta^*\in (0, 1/2)$, it is true that
\begin{equation}
    \label{eq p3 hypo}
     \DNwms\leq A_2\left(\delta^* (Q^*)^{n+R} +\left(Q^*\right)^{\alpha}\left(\left(\log 4Q^*\right)\left(\frac{n+2R}{n+2}\right)\right)^{\atwo}\right).
\end{equation}
Then there exists a positive constant $C_2$ depending only on $w$ and $\cM$ such that for all $s\in \{1, \ldots, R\}$, $Q\geq 1$ and $\bdel\in (0, 1/2)^R$, we have 
\begin{equation}
    \label{eq p3 concl}
    {\mathfrak{N}}^s_{w, \cM}(Q, \bdel)\leq C_2\delp Q^{n+1}+C_2A_2\left( \delps Q^{\frac{(n+1)(R-1)}{R}} Q^{\frac{\tilde{\beta}}{R}} +Q^{\tilde{\beta}}\left(\log 4Q\right)^{\atwoR}\right);
\end{equation}
with 
\begin{equation}
    \label{eq beta beta rel}
    \tilde{\beta}=n+1-\frac{n R}{2 \alpha- n}\in \left[\frac{n(n+R+1)}{n+2R}, n+1\right].
\end{equation}
\end{proposition}
\begin{proof}[Proof of Proposition \ref{prop 1}, given Propositions \ref{prop dual from og} and \ref{prop og from dual}]
Using the hypothesis \eqref{eq p1 hypo}, we apply Proposition \ref{prop dual from og} to obtain \eqref{eq p2 concl} with $A_1=A$, $a_1=a$ and 
$$\alpha =\max\left(n+R-\frac{n }{2 \beta- n}, n+R-1-\frac{2}{n}\right)=n+R-\frac{n}{2\beta-n}.$$
Indeed, since $$\beta\geq \beta_{\textrm{st}}\geq \frac{n(n+1)}{n+2},$$
we have
$$n+R-\frac{n }{2 \beta- n}\geq n+R-\frac{n }{\frac{2n(n+1)}{n+2}- n}=n+R-1-\frac{2}{n}.$$
This in turn implies that the hypothesis of Proposition \ref{prop og from dual} is true with the same $\alpha$, and with $a_2=a$ and $A_2=\max\{1, \mathfrak{C}_0C_1A\}$. Here $\mathfrak{C}_0$ is the constant from condition \eqref{eq curv est}. Applying Proposition \ref{prop og from dual} next, we conclude that for all $s\in \{1, \ldots, R\}$, $Q\geq 1$ and $\bdel\in (0, 1/2)^R$, 
\begin{equation*}
{\mathfrak{N}}^s_{w, \cM}(Q, \bdel)\leq C_0'\delp Q^{n+1}+C_0'A_2\left( \sum_{r=1}^R \delpr Q^{\frac{(n+1)(R-1)}{R}} Q^{\frac{\tilde{\beta}}{R}} +Q^{\tilde{\beta}}(\log 4Q)^{(a+1)R}\right)    
\end{equation*}
where $C_0'$ depends only on $w$ and $\cM$ and
$$\tilde{\beta}=n+1-\frac{nR}{2\alpha-n}n+1-\frac{nR}{2\left(n+R-\frac{n}{2\beta-n}\right)-n}=n+1-\frac{nR}{2R+n\left(1-\frac{2}{2\beta-n}\right)}.$$
Summing up in $s$ and using $\eqref{eq split codim}$, we obtain \eqref{eq p1 concl} with $C_0=R2^R C_0'$.
\end{proof}

Both Propositions \ref{prop dual from og} and \ref{prop og from dual} rely on a combination of projective duality and van der Corput's B-process to pass from $\mathcal{M}$ to the family of dual hypersurfaces, and back. The following propositions make this connection precise. 

We state the dual version first: for passage from the sum of the dual weights to rational point count in the neighborhood of $\cM$. This should be compared to \cite[Proposition 5.3 and \S 6.3]{schindler2022density}, where the original counting problem is projected to a lower dimensional one associated to a family of hypersurfaces in $\mathbb{R}^{n+1}$, with a trivial summing up in the remaining $R-1$ directions. 

\begin{proposition}[Dual van der Corput B-Process]
    \label{prop dual vdc b}
     Let $s\in \{1, \ldots, R\}$. For all $Q^*\geq 1$ and $\delta^*\in (0, 1/2)$ with $D:=\left\lfloor \frac{1}{2\delta^*}\right\rfloor$, we have
    \begin{align}
    \label{eq dual vdc b}
     \DNwms & \ll\delta^{*}(Q^*)^{n+R}+
     \frac{(Q^*)^{\frac{n}{2}}}{D}\sum_{d=1}^D d^{-\frac{n}{2}}\sum_{\bk\in \mathbb{Z}^n}{w\kq}\left(\prod_{r=1}^R\min\left(\|df_r(\bk/d)\|^{-1}, Q^*\right)\right)\nonumber\\&+(Q^*)^{\frac{n}{2}+R-1}D^{\frac{n}{2}-1},
    \end{align}
with the implicit constant depending only on $w$ and $\mathcal{M}$.  
\end{proposition}

The final proposition gives an upper bound for ${\mathfrak{N}}^s_{w, \cM}(Q, \bdel)$ in terms of rational points around the dual family of hypersurfaces.

\begin{proposition}[van der Corput B-Process for $\cM$]
    \label{prop og vdc b}
    Let $s\in \{1, \ldots, R\}$. For $Q\geq 1$ and $\bdel\in (0, 1/2)^R$, let ${\mathfrak{N}}^s_{w, \cM}(Q, \bdel)$ be as defined in \eqref{eq Nog 11 sum} with $J_s:=\left\lfloor \frac{1}{2\delta_s}\right\rfloor$. 
    We have
    \begin{align}
        \label{eq og vdc b}
        &{\mathfrak{N}}^s_{w, \cM}(Q, \bdel)\ll\delp Q^{n+1}\\&+ Q^{\frac{n}{2}}
        \sum_{\bj\in \mathcal{J}^s(J_s)}\left(\prod_{r=1}^R \frac{1}{j_r+1}\right)\sum_{\bk\in \mathbb{Z}^n}\frac{w_{s, \bj}^*\kjs}{\sqrt{\left|\det\, H_{F_{s, \bj}}\left((\nabla F_{s, \bj})^{-1}(\bk/j_s)\right)\right|}}j_s^{-\frac{n}{2}}\min\left(\|j_sF^*_{s, \bj}(\bk/j_s)\|^{-1}, Q\right)\nonumber\\
        &+J_s^{\frac{n}{2}-1}Q^{\frac{n}{2}}\left(\log 4J_1\right)^R.
    \end{align}
The implicit constant depends only on $w$ and $\mathcal{M}$.  
\end{proposition}
We shall present the proofs of Propositions \ref{prop dual from og}-\ref{prop og vdc b} for $s=1$. The other cases can be reduced to this one by a relabelling of the variables $j_1, \ldots, j_r$. Consequently, in the subsequent sections, we shall suppress notation and omit the parameter $s$. In other words, $\mathcal{J}^s$ shall be denoted by $\mathcal{J}$, $\mathcal{J}^s_b$ by $\mathcal{J}_b$, $F_{s, \bj}$ by $F_{\bj}$, $w_{s, \bj}$ by $w_{\bj}$, $\mathcal{R}_{s, \bj}$ by $\mathcal{R}_{\bj}$, $\mathcal{V}_{s, \bj}$ by $\mathcal{V}_{\bj}$, and so on.

\section{Proof of Proposition \ref{prop dual from og} using Proposition \ref{prop dual vdc b}}
\label{sec dual from og}
Let $D:=\left\lfloor \frac{1}{2\delta^*}\right\rfloor$. By Proposition \ref{prop dual vdc b}, we have
\begin{align}
\label{eq p4 in p2}
{\mathfrak{N}}^{*, 1}_{w, \cM}(Q^*, \delta^*)
&\ll \delta^{*}(Q^*)^{n+R}+
\frac{(Q^*)^{\frac{n}{2}}}{D}\sum_{d=1}^D d^{-\frac{n}{2}}\sum_{\bk\in \mathbb{Z}^n}{w\left(\frac{\bk}{d}\right)}\left(\prod_{r=1}^R\min\left(\|df_r(\bk/d)\|^{-1}, Q^*\right)\right)\nonumber\\
&+(Q^*)^{\frac{n}{2}+R-1}D^{\frac{n}{2}-1}
\end{align}
with the implicit constant depending only on $w$ and $\cM$. To deal with the second term on the right, we employ dyadic decomposition based on the size of $\|df_r(\bk/d)\|$ with respect to $Q^*$ to obtain
\begin{align}
\label{eq p2 1}
\sum_{\bk\in \mathbb{Z}^n}{w\left(\frac{\bk}{d}\right)}\left(\prod_{r=1}^R\min\left(\|df_r(\bk/d)\|^{-1}, Q^*\right)\right)\leq  
(Q^*)^R\sum_{\substack{\mathbf{i}\in \mathbb{Z}^R_{\geq 0}:\\\|\mathbf{i}\|_{\infty}\leq \frac{\log 4Q^*}{\log 2}}}2^{-\sum_{r=1}^Ri_r}\sum_{\substack{\bk\in \mathbb{Z}^n\\ \|df_r(\frac{\bk}{d})\|\leq \frac{2^{i_r+1}}{Q^*} \\1\leq r\leq R}} w\left(\frac{\bk}{d}\right).
\end{align}
Our induction hypothesis \eqref{eq p2 hypo} implies that
\begin{align*}
    &\sum_{d=1}^D\sum_{\substack{\bk\in \mathbb{Z}^n\\ \|df_r(\frac{\bk}{d})\|\leq \frac{2^{i_r+1}}{Q^*}: 1\leq r\leq R }} w\left(\frac{\bk}{d}\right)=\mathfrak{N}_{w, \cM}\left(D, \frac{2^{i_1}}{Q^*}, \ldots, \frac{2^{i_R}}{Q^*}\right)\\
    &\leq A_1\left(2^{\sum_{r=1}^Ri_r}(Q^*)^{-R} D^{n+1}+\sum_{s=1}^R  \frac{2^{\sum_{r=1}^Ri_r}}{2^{i_s}} (Q^*)^{-R+1} D^{\frac{(n+1)(R-1)}{R}}D^{\frac{\beta}{R}}+D^{\beta}(\log 4D)^{\aone}\right).
\end{align*}
Partial summation in the $d$ variable then lets us conclude that 
\begin{align*}
   &\sum_{d=1}^D d^{-\frac{n}{2}}\sum_{\substack{\bk\in \mathbb{Z}^n \\ \|df_r(\frac{\bk}{d})\|\leq \frac{2^{i_r+1}}{Q^*} :1\leq r\leq R}} w\left(\frac{\bk}{d}\right)\leq \\&{A_1}{ D^{-\frac{n}{2}}}\Bigg(2^{\sum_{r=1}^Ri_r}(Q^*)^{-R} D^{n+1}+ \sum_{s=1}^R  \frac{2^{\sum_{r=1}^Ri_r}}{2^{i_s}} (Q^*)^{-R+1} D^{\frac{(n+1)(R-1)+\beta}{R}}+D^{\beta}(\log 4D)^{\aone}\Bigg).
\end{align*}
We combine the above with \eqref{eq p2 1} to get
\begin{align}
\label{eq p2 2}
&\sum_{d=1}^Dd^{-\frac{n}{2}}\sum_{\bk\in \mathbb{Z}^n}{w\left(\frac{\bk}{d}\right)}\left(\prod_{r=1}^R\min\left(\|df_r(\bk/d)\|^{-1}, Q^*\right)\right)\nonumber\\
&\leq  A_1D^{-\frac{n}{2}}(Q^*)^R\sum_{\substack{\mathbf{i}\in \mathbb{Z}^R_{\geq 0}:\\\|\mathbf{i}\|_{\infty}\leq \frac{\log 4Q^*}{\log 2}}}2^{-\sum_{i=1}^Ri_r}\Bigg(2^{\sum_{r=1}^Ri_r}(Q^*)^{-R} D^{n+1}+\sum_{s=1}^R  \frac{2^{\sum_{r=1}^Ri_r}}{2^{i_s}} (Q^*)^{-R+1} D^{\frac{(n+1)(R-1)+\beta}{R}}\nonumber\\&+D^{\beta}(\log 4D)^{\aone}\Bigg)\nonumber\\
&\leq A_1 \left(\left(\frac{\log 4Q^*}{\log 2}\right)^R D^{\frac{n}{2}+1} + R Q^*\left(\frac{\log 4Q^*}{\log 2}\right)^{R-1}  D^{\frac{(n+1)(R-1)+\beta}{R}-\frac{n}{2}}+D^{\beta-\frac{n}{2}}(\log 4D)^{\aone}(Q^*)^{R}\right).
\end{align}
Plugging the above in \eqref{eq p4 in p2} lets us conclude that $\DNwm$ is bounded from above by a positive constant times
\begin{align*}
&\delta^{*}(Q^*)^{n+R}+A_1\frac{(Q^*)^{\frac{n}{2}}}{D}\Big(\left(\log 4Q^*\right)^R D^{\frac{n}{2}+1} +  Q^*\left(\log 4Q^*\right)^{R-1}  D^{\frac{(n+1)(R-1)+\beta}{R}-\frac{n}{2}}\nonumber\\&+(Q^*)^RD^{\beta-\frac{n}{2}}(\log 4D)^{\aone}\Big)+
(Q^*)^{\frac{n}{2}+R-1}D^{\frac{n}{2}-1}.
\end{align*}
Recalling that $D=\left\lfloor \frac{1}{2\delta^*}\right\rfloor$, we get
\begin{align}
\label{eq p2 bf ind}
\DNwm&\leq C\delta^{*}(Q^*)^{n+R}+CA_1\Big(\left(\log 4Q^*\right)^R(Q^*)^{\frac{n}{2}} (\delta^*)^{-\frac{n}{2}} +(Q^*)^{\frac{n}{2}+R}(\delta^*)^{\frac{n}{2}+1-\beta}\\\nonumber&\times(\log (4/\delta^*))^{\aone}+\left(\log 4Q^*\right)^{R-1}(Q^*)^{\frac{n}{2}+1} (\delta^*)^{\frac{n}{2}+1-\frac{(n+1)(R-1)+\beta}{R}}  \Big)\\\nonumber
&+C(Q^*)^{\frac{n}{2}+R-1}(\delta^*)^{1-\frac{n}{2}}.   
\end{align}
Observe that all terms except the first on the right hand side of \eqref{eq p2 bf ind} involve a negative power of $\delta^*$. To see this for the penultimate term, we estimate, first for $R\geq 2$,
\begin{equation*}
    \frac{n}{2}+1-\frac{(n+1)(R-1)+\beta}{R}= \frac{n+1-\beta}{R}-\frac{n}{2}\leq \frac{1-\beta}{2}\leq 0.
\end{equation*}
For $R=1$, the power of $\delta^*$ on the penultimate term reduces to
\begin{equation*}
    \frac{n}{2}+1-\beta\leq 0, 
\end{equation*}
as $\beta\geq n$ in this case. The first term on the right in \eqref{eq p2 bf ind} is the expected main term. However, it is only going to dominate the sum of the other terms on the right above the threshold $\delta^*\geq (Q^*)^{-\frac{n}{2\beta-n}}$. By using the monotonicity of the counting function in $\delta^*$, it is always possible to inflate $\delta^*$ to this scale. The evaluation of the right hand side at $\delta^*= (Q^*)^{-\frac{n}{2\beta-n}}$ will give us the order of the error term. 

We first record a few small calculations which shall be of use later. Since $\beta\geq \frac{n(n+R+1)}{n+2R}$, we have
\begin{equation}
\label{eq est beta term}
\frac{n}{2\beta-n}\leq \frac{n}{\frac{2n(n+R+1)}{n+2R}-n}\leq \frac{n+2R}{n+2}.
\end{equation}
It is also straightforward to verify that with the aforementioned lower bound on $\beta$,
\begin{equation}
    \label{eq p2 est}
    \frac{n}{2}\left(1+\frac{n}{2\beta-n}\right)\leq n+R-\frac{n}{2\beta-n}.
\end{equation}
Further, since $\beta\geq \frac{n(n+1)}{n+2}$, we have
\begin{equation}
    \label{eq prob est}
    \left(\frac{n}{2\beta-n}\right)\left(\frac{n}{2}-1\right)\leq \frac{n\left(\frac{n}{2}-1\right)}{\frac{2n(n+1)}{n+2}-n}=\frac{n+2}{n}\left(\frac{n}{2}-1\right)=\frac{n}{2}-\frac{2}{n}.
\end{equation}

We now return to \eqref{eq p2 bf ind}, and consider two cases based on the size of $\delta^*$ with respect to $Q^*$. If $\delta^*\geq (Q^*)^{-\frac{n}{2\beta-n}}$, then we can estimate
\begin{equation}
\label{eq p2 t2}
(Q^*)^{\frac{n}{2}}(\delta^*)^{-\frac{n}{2}}\leq (Q^*)^{\frac{n}{2}\left(1+\frac{n}{2\beta-n}\right)}\stackrel{\eqref{eq p2 est}}{\leq}(Q^*)^{n+R-\frac{n}{2\beta-n}}.
\end{equation}
Next, we deal with the last term on the right hand side in \eqref{eq p2 bf ind} as follows 
\begin{align}
\label{eq p2 l term}
(Q^*)^{\frac{n}{2}+R-1}(\delta^*)^{1-\frac{n}{2}}&\leq (Q^*)^{\frac{n}{2}+R-1+\left(\frac{n}{2\beta-n}\right)\left(\frac{n}{2}-1\right)}\stackrel{\eqref{eq prob est}}{\leq}(Q^*)^{\frac{n}{2}+R-1+\frac{n}{2}-\frac{2}{n}}\nonumber
\\&=(Q^*)^{n+R-1-\frac{2}{n}}. 
\end{align}
We now estimate the two remaining middle terms which determine the relation of $\alpha$ with respect to $\beta$. The first one is easy
\begin{align}
\label{eq p2 ind main}
(Q^*)^{\frac{n}{2}+R}(\delta^*)^{\frac{n}{2}+1-\beta}(\log 4(1/\delta^*))^{\aone}&\leq (Q^*)^{\frac{n}{2}+R+\left(\frac{n}{2\beta-n}\right)\left(\beta-\frac{n}{2}-1\right)}\left((\log 4Q^*)\left(\frac{n}{2\beta-n}\right)\right)^{\aone}\nonumber
\\&\stackrel{\eqref{eq est beta term}}{\leq} (Q^*)^{n+R-\frac{n}{2\beta-n}}\left((\log 4Q^*)\left(\frac{n+2R}{n+2}\right)\right)^{\aone}.
\end{align}
We also have
\begin{align}
    \label{eq p2 betaR term}
    (Q^*)^{\frac{n}{2}+1} (\delta^*)^{\frac{n}{2}+1-\frac{(n+1)(R-1)+\beta}{R}}&=\left((Q^*)^{\frac{n}{2}}(\delta^*)^{-\frac{n}{2}}\right)^{\frac{R-1}{R}} \left((Q^*)^{\frac{n}{2}+R}(\delta^*)^{\frac{n}{2}+1-\beta}\right)^{\frac{1}{R}}\nonumber\\&\stackrel{\eqref{eq p2 t2}+\eqref{eq p2 ind main}}{\leq}\left((Q^*)^{\frac{(R-1)}{R}+\frac{1}{R}}\right)^{{n+R-\frac{n}{2\beta-n}}}=(Q^*)^{{n+R-\frac{n}{2\beta-n}}}.
\end{align}
Combining \eqref{eq p2 bf ind}, \eqref{eq p2 t2}, \eqref{eq p2 l term}, \eqref{eq p2 ind main} and \eqref{eq p2 betaR term}, yields
$$ \DNwm \leq C'\delta^* (Q^*)^{n+R}+C'A_1\left(Q^*\right)^{\alpha}\left(\left(\log 4Q^*\right)\left(\frac{n+2R}{n+2}\right)\right)^{\aone}$$
for $\delta^*\geq (Q^*)^{-\frac{n}{2\beta-n}}$. 

In the complementary case when $\delta^*\leq (Q^*)^{-\frac{n}{2\beta-n}}$,  
we use  monotonicity of the function $\DNwm$ to deduce that
\begin{equation}
  \DNwm\leq \mathfrak{N}^{*, 1}_{w, \cM}\left(Q^*, (Q^*)^{-\frac{n}{2\beta-n}}\right).  
\end{equation}
To estimate the right hand side, we use \eqref{eq p2 bf ind} with $\delta^*=(Q^*)^{-\frac{n}{2\beta-n}}$. Applying the same arguments as in \eqref{eq p2 t2}-\eqref{eq p2 betaR term}, and evaluating the first term directly, gives
\begin{align}
    {\mathfrak{N}}^{*, 1}_{w, \cM}\left(Q^*, (Q^*)^{-\frac{n}{2\beta-n}}\right)&\leq C(Q^*)^{-\frac{n}{2\beta-n}+n+R}+CA_1(Q^*)^{\alpha}\left(\left(\log 4Q^*\right)\left(\frac{n+2R}{n+2}\right)\right)^{\aone}\nonumber\\
    &\leq C''A_1(Q^*)^{\alpha}\left(\left(\log 4Q^*\right)\left(\frac{n+2R}{n+2}\right)\right)^{\aone}.
\end{align}
This establishes \eqref{eq p2 concl} also in the case when  $\delta^*\leq (Q^*)^{-\frac{n}{2\beta-n}}$, and thus finishes the proof of Proposition \ref{prop dual from og}. The constant $C_1$ in \eqref{eq p2 concl} can be taken to be $\max(C', C'')$ and depends only on $w$ and $\cM$. 

\section{Proof of Proposition \ref{prop og from dual} using Proposition \ref{prop og vdc b}}
\label{sec og from dual}
Recall that $J_1:=\left\lfloor \frac{1}{2\delta_1}\right\rfloor$. We begin by applying Proposition \ref{prop og vdc b} to obtain
    \begin{align}
        \label{eq p5 in p3}
        &{\mathfrak{N}}^1_{w, \cM}(Q, \bdel)\ll\delp Q^{n+1}+J_1^{\frac{n}{2}-1}Q^{\frac{n}{2}}\left(\log 4J_1\right)^R\nonumber\\&+ Q^{\frac{n}{2}}
        \sum_{\bj\in \mathcal{J}(J_1)}\left(\prod_{r=1}^R \frac{1}{j_r+1}\right)\sum_{\bk\in \mathbb{Z}^n}\frac{w_{\bj}^*\kj}{\sqrt{\det\, H_{F_{\bj}}\left((\nabla F_{\bj})^{-1}(\bk/j_1)\right)}}j_1^{-\frac{n}{2}}\min\left(\|j_1F^*_{\bj}(\bk/j_1)\|^{-1}, Q\right).
    \end{align}
To deal with the last term on the right, we use dyadic decomposition based on the size of $\|j_1F_{\bj}^*(\bk/j_1)\|$ with respect to $Q$ to split
\begin{align}
\label{eq p3 1}
&\sum_{\bk\in \mathbb{Z}^n}\frac{w_{\bj}^*\kj}{\sqrt{\left|\det\, H_{F_{\bj}}\left((\nabla F_{\bj})^{-1}(\bk/j_1)\right)\right|}}\min\left(\|j_1F^*_{\bj}(\bk/j_1)\|^{-1}, Q\right)\nonumber\\&\leq  
Q\sum_{\substack{\bk\in \mathbb{Z}^n\\\|j_1F^*_{\bj}(\bk/j_1)\|<Q^{-1}}}\frac{w_{\bj}^*\kj}{{\sqrt{\left|\det\, H_{F_{\bj}}\left((\nabla F_{\bj})^{-1}(\bk/j_1)\right)\right|}}}
\\\nonumber&+\sum_{0\leq i\leq \frac{\log 4Q}{\log 2}} 2^{-i}Q\sum_{\substack{\bk\in \mathbb{Z}^n\\\frac
{2^i}{Q}<\|j_1F^*_{\bj}(\bk/j_1)\|\leq \frac{2^{i+1}}{Q}}}\frac{w_{\bj}^*\kj}{{\sqrt{\left|\det\, H_{F_{\bj}}\left((\nabla F_{\bj})^{-1}(\bk/j_1)\right)\right|}}}.
\end{align}
For $-1\leq i\leq \frac{\log 4Q}{\log 2}$, our induction hypothesis \eqref{eq p3 hypo} implies that
\begin{align}
   & \sum_{\bj\in \mathcal{J}(J_1)}
   \sum_{\substack{\bk\in \mathbb{Z}^n\\\|j_1F^*_{\bj}(\bk/j_1)\|\leq \frac{2^{i+1}}{Q}}}\frac{w_{\bj}^*\kj}{{\sqrt{\left|\det\, H_{F_{\bj}}\left((\nabla F_{\bj})^{-1}(\bk/j_1)\right)\right|}}}\nonumber\\
   &={\mathfrak{N}}^{*, 1}_{w, \cM}\left(J_1, \frac{2^{i+1}}{Q}\right)
   \leq A_2\left(2^{i+1}Q^{-1} J_1^{n+R} +J_1^{\alpha}\left(\left(\log 4J_1\right)\left(\frac{n+2R}{n+2}\right)\right)^{\atwo}\right).
   \label{eq ndual est}
\end{align}
Partial summation in the $j_r$ variables, keeping in mind that 
$j_r\leq j_1\leq J_1$, then yields 
\begin{align*}
   & \sum_{\bj\in \mathcal{J}(J_1)}
   \left(\prod_{r=1}^R \frac{1}{j_r+1}\right)\sum_{\substack{\bk\in \mathbb{Z}^n\\\|j_1F^*_{\bj}(\bk/j_1)\|\leq \frac{2^{i+1}}{Q}}}\frac{w_{\bj}^*\kj j_1^{-\frac{n}{2}}}{{\sqrt{\left|\det\, H_{F_{\bj}}\left((\nabla F_{\bj})^{-1}(\bk/j_1)\right)\right|}}}\\
   &\leq \sum_{0\leq s_2, \ldots, s_R\leq \frac{\log J_1}{\log 2}+1} \sum_{\substack{1\leq j_1\leq J_1\\j_r\in [2^{s_r-1}, 2^{s_r}]\\(2\leq r\leq R)}} \left(\prod_{r=2}^R 2^{-s_r}\right)j_1^{-\frac{n}{2}-1}\sum_{\substack{\bk\in \mathbb{Z}^n\\\|j_1F^*_{\bj}(\bk/j_1)\|\leq \frac{2^{i+1}}{Q}}}\frac{w_{\bj}^*\kj }{{\sqrt{\left|\det\, H_{F_{\bj}}\left((\nabla F_{\bj})^{-1}(\bk/j_1)\right)\right|}}}\\
   &\leq CA_2(\log 4J_1)^{R}J_1^{-\frac{n}{2}}\left(2^{i+1}Q^{-1} J_1^{n} +J_1^{\alpha-R}\left(\left(\log 4J_1\right)\left(\frac{n+2R}{n+2}\right)\right)^{\atwo}\right).
\end{align*}

Combining the decomposition in \eqref{eq p3 1} with the above estimate, we get
\begin{align*}
    &Q^{\frac{n}{2}}  \sum_{\bj\in \mathcal{J}(J_1)}
    \left(\prod_{r=1}^R \frac{1}{j_r+1}\right)\sum_{\bk\in \mathbb{Z}^n}\frac{w_{\bj}^*\kj j_1^{-\frac{n}{2}}}{{\sqrt{\left|\det\, H_{F_{\bj}}\left((\nabla F_{\bj})^{-1}(\bk/j_1)\right)\right|}}}\min\left(\|j_1F^*_{\bj}(\bk/j_1)\|^{-1}, Q\right)
    \\&\leq CA_2Q^{\frac{n}{2}+1} (\log 4J_1)^{R}J_1^{-\frac{n}{2}}\sum_{0\leq i\leq \frac{\log 4Q}{\log 2}}2^{-i}\left(2^{i+1}Q^{-1} J_1^{n} +J_1^{\alpha-R}\left(\left(\log 4J_1\right)\left(\frac{n+2R}{n+2}\right)\right)^{\atwo}\right)\\
    &\leq CA_2\left((\log 4J_1)\left(\frac{n+2R}{n+2}\right)\right)^{b} \left((\log 4Q)Q^{\frac{n}{2}}J_1^{\frac{n}{2}}+Q^{\frac{n}{2}+1}J_1^{\alpha-\frac{n}{2}-R}\right),
\end{align*}
where $$b:=\atwoR.$$
Plugging the above in \eqref{eq p5 in p3} lets us conclude that 
\begin{equation*}
{\mathfrak{N}}^1_{w, \cM}(Q, \bdel)\leq C\delp Q^{n+1}+CA_2\left(\left(\log 4J_1\right)\left(\frac{n+2R}{n+2}\right)\right)^b \left((\log 4Q)Q^{\frac{n}{2}}J_1^{\frac{n}{2}}+Q^{\frac{n}{2}+1}J_1^{\alpha-\frac{n}{2}-R}\right).  
\end{equation*}
Recalling that $J_1=\left\lfloor \frac{1}{2\delta_1}\right\rfloor$, we get
\begin{equation}
\label{eq p3 b4 ind}
{\mathfrak{N}}^1_{w, \cM}(Q, \bdel)\leq C\delp Q^{n+1}+ CA_2\left((\log 4\delta_1^{-1}) \left(\frac{n+2R}{n+2}\right)\right)^b\left((\log 4Q)Q^{\frac{n}{2}}\delta_1^{-\frac{n}{2}}+Q^{\frac{n}{2}+1}\delta_1^{\frac{n}{2}+R-\alpha}\right).  
\end{equation}
Just as in the proof of Proposition \ref{prop dual from og}, the first term on the right is the expected main term. However, it is only going to dominate the sum of the other two terms on the right above the threshold $\delta_1\geq Q^{-\frac{n}{2\alpha-n}}$. Our strategy is again to inflate $\delta_1$ to this scale by using the monotonicity of the counting function in $\delta_1$. The evaluation of the right hand side at $\delta_1=Q^{-\frac{n}{2\alpha-n}}$ will determine the order of the error term. We make a couple of quick observations. Since $\alpha\geq \frac{n(n+R+1)}{n+2}$, we have
\begin{equation}
\label{eq est alpha term}
\frac{n}{2\alpha-n}\left(\frac{n}{2}+R-\alpha\right)\leq \frac{n}{\frac{2n(n+R+1)}{n+2}-n}= \frac{n+2}{n+2R}.
\end{equation}
It is also straightforward to verify that with the aforementioned lower bound on $\alpha$, we have
\begin{equation}
    \label{eq alph est}
    \frac{n}{2}\left(1+\frac{n}{2\alpha-n}\right)\leq n+1-\frac{nR}{2\alpha-n}.
\end{equation}
We now return to \eqref{eq p3 b4 ind} and consider two cases, depending on the size of $\delta_1$ with respect to $Q^{-\frac{n}{2\alpha-n}}$. In the case when $\delta_1\geq Q^{-\frac{n}{2\alpha-n}}$, we can bound
\begin{align*}
    &{\mathfrak{N}}^1_{w, \cM}(Q, \bdel)\\
    &\leq C\delp Q^{n+1}+ CA_2\left((\log 4Q^{\frac{n}{2\alpha-n}})\left(\frac{n+2R}{n+2}\right)\right)^{b} \left((\log 4Q)Q^{\frac{n}{2}\left(1+\frac{n}{2\alpha-n}\right)}+Q^{\frac{n}{2}+1-\frac{n}{2\alpha-n}\left(\frac{n}{2}+R-\alpha\right)}\right)\\
    &\stackrel{\eqref{eq est alpha term}+\eqref{eq alph est}}{\leq}  C\delp Q^{n+1}+ CA_2(\log4Q)^{b}\left(\frac{n+2R}{n+2}\right)^{-b}\left(\frac{n+2R}{n+2}\right)^{b} Q^{n+1-\frac{nR}{2\alpha-n}}\\
    &\leq  C'\delp Q^{n+1}+ C'A_2\left(\log 4Q\right)^{\atwoR} Q^{\tilde{\beta}}.
\end{align*}

This establishes \eqref{eq p3 concl} when $\delta_1\geq Q^{-\frac{n}{2\alpha-n}}$. In the complementary case, we use montonicity of ${\mathfrak{N}}^1_{w, \cM}(Q, \bdel)$ as a function of $\delta_1$ to deduce that
\begin{align*}
&{\mathfrak{N}}^1_{w, \cM}(Q, \bdel)\\
&\leq {\mathfrak{N}}^1_{w, \cM}\left(Q, Q^{-\frac{n}{2\alpha-n}}, \delta_2, \ldots, \delta_R\right)\\
&\stackrel{\eqref{eq p3 b4 ind}}{\leq} C\bdel_1^\times Q^{n+1-\frac{n}{2\alpha-n}}+ CA_2(\log 4Q^{\frac{n}{2\alpha-n}})^{b}\left(\frac{n+2R}{n+2}\right)^{b}  \Big((\log 4Q)Q^{\frac{n}{2}\left(1+\frac{n}{2\alpha-n}\right)}\\&+Q^{\frac{n}{2}+1-\frac{n}{2\alpha-n}\left(\frac{n}{2}+R-\alpha\right)}\Big)\\
&\stackrel{\eqref{eq est alpha term}+\eqref{eq alph est}}{\leq} C''\delpone Q^{\frac{(R-1)(n+1)}{R}}Q^{\frac{1}{R}\left(n+1-\frac{nR}{2\alpha-n}\right)}+ C''A_2(\log 4Q)^{b} Q^{n+1-\frac{nR}{2\alpha-n}}\\
&=  C''\delpone Q^{\frac{(R-1)(n+1)}{R}}Q^{\frac{\tilde{\beta}}{R}}+ C''A_2\left(\log 4Q\right)^{\atwoR} Q^{\tilde{\beta}}.
\end{align*}
This establishes \eqref{eq p3 concl} also in the case when $\delta_1\leq Q^{-\frac{n}{2\alpha-n}}$, and finishes the proof.

\section{Proof of Proposition \ref{prop dual vdc b}}
\label{sec dual vdc b}
For convenience, we recall the definition
\begin{equation}
    \label{eq def dual weight Hessian 2}
    {\mathfrak{N}}^{*, 1}_{w, \cM}(Q^*, \delta^*):=
    \sum_{\bj\in\mathcal{J}(Q^*)}\sum_{\substack{\ba\in \mathbb{Z}^n\\\|j_1 F^*_{\bj}(\ba/j_1)\|<\delta^{*}}}\frac{w_{\bj}^*\aj}{{\sqrt{\left|\det\, H_{F_{\bj}}\left((\nabla F_{\bj})^{-1}(\ba/j_1)\right)\right|}}}.
\end{equation}
For technical reasons, it will be helpful to work with a dyadic version of the above counting function. For $\ell\in\mathbb{Z}_{>0}$, let
\begin{equation}
\label{def index set Jell}
\mathcal{J}_{\ell}^1:=\mathcal{J}_{\ell}=\{\bj\in \mathbb{Z}_{\geq 0}^R: 2^{\ell-1}\leq j_1 < 2^{\ell}, \max_{r\neq 1}j_r \leq 2^{\ell}\}.    
\end{equation}
The counting function in \eqref{eq def dual weight Hessian 2} can be dominated by its dyadic version given by
\begin{equation}
    \label{def dual N dyad}
    {N}^{*, 1}_{w, \cM}(Q^*, \delta^*):=\sum_{\ell=1}^{\lgQ*}\sum_{\bj\in \mathcal{J}_{\ell}}\sum_{\substack{\ba\in \mathbb{Z}^n\\\|j_1 F^*_{\bj}(\ba/j_1)\|<\delta^{*}}}\frac{w_{\bj}^*\aj}{\sqrt{\left|\det\, H_{F_{\bj}}\left((\nabla F_{\bj})^{-1}(\ba/j_1)\right)\right|}}.
\end{equation}
For each $\ell\in \mathbb{Z}_{>0}$, 
$\mathcal{J}_{\ell}$ is a subset of $\mathcal{J}_b$ (as defined in \eqref{eq Jbig index set} with $s=1$).   Consequently, the condition \eqref{eq curv est} continues to hold and $\nabla F_{\bj}$ is still a diffeomorphism on $\mathscr{D}$ and $U$ for all $\bj\in \mathcal{J}_{\ell}$. The dyadic counting function above is therefore well defined. 

We shall establish \eqref{eq dual vdc b} with ${\mathfrak{N}}^{*, 1}_{w, \cM}(Q^*, \delta^*)$ replaced by ${N}^{*, 1}_{w, \cM}(Q^*, \delta^*)$.
\begin{lemma}
\label{lem dual main term}    
Let $D:=\left\lfloor \frac{1}{2\delta^*}\right\rfloor.$ Then
\begin{equation}
    \label{eq lem dual main term}
    {{N}}^{*, 1}_{w, \cM}(Q^*, \delta^*)\ll(\delta^*)(Q^*)^{n+R}+\left|\sum_{\ell=1}^{\lgQ*}\sum_{\bj\in \mathcal{J}_{\ell}}\sum_{d=1}^D \frac{D-d}{D^2}\sum_{\substack{\bk\in \mathbb{Z}^n}}j_1^nI^*(d, \bj, \bk)\right|,
\end{equation}
with \begin{equation}
    \label{def dual osc int}
    I^*(d, \bj, \bk):=\int_{\mathbb{R}^n} \frac{w_{\bj}^*(\bz)e\left(dj_1\left(F_{\bj}^*(\bz)-\bk\cdot\bz\right)\right)}{\sqrt{\left|\det\, H_{F_{\bj}}\left((\nabla F_{\bj})^{-1}(\bz)\right)\right|}} \, d\bz.
\end{equation}
\end{lemma}
\begin{proof}
Let $\mathscr{F}_D: \mathbb{R}\to [0, 1]$ be the Fej\'er kernel of degree $D$ given by
\begin{equation}
    \label{def Fejer kernel}
    \mathscr{F}_D(\theta):=\sum_{d=-D}^{D}\frac{D-|d|}{D^2}e(d\theta)=\left(\frac{\sin\left(\pi D\theta\right)}{D\sin \left(\pi \theta\right)}\right)^2.
\end{equation}
Recall from \eqref{eq fejer char est} that this function has the property 
\begin{equation}
    \mathds{1}_{\delta^*}(\theta)\leq \frac{\pi^2}{4}\mathscr{F}_D(\theta)\qquad\,\text{ for all } \theta\in \mathbb{R}.
\end{equation}
Therefore ${{N}}^{*, 1}_{w, \cM}(Q^*, \delta^*)$ can be dominated by a positive constant times 
\begin{equation}
    \label{eq dual N expsum}
 \left|\sum_{\ell=1}^{\lgQ*}\sum_{\bj\in \mathcal{J}_{\ell}}\sum_{\substack{\ba\in \mathbb{Z}^n}}\frac{w_{\bj}^*\aj}{\sqrt{\left|\det\, H_{F_{\bj}}\left((\nabla F_{\bj})^{-1}(\ba/j_1)\right)\right|}}\sum_{d=-D}^D\frac{D-|d|}{D^2}e\left(dj_1 F^*_{\bj}(\ba/j_1)\right)\right|.
\end{equation}
For all $\bj\in \mathcal{J}_b$ (see \eqref{eq Jbig index set}), the sets $$\supp\, w_{\bj}^*:=V_{\bj}$$ are contained in a compact set $\mathscr{L}\subseteq \mathbb{R}^n$. Thus the contribution from the term corresponding to $d=0$ in \eqref{eq dual N expsum} is given by
\begin{equation}
    \label{eq dual zero freq}
    \frac{1}{D}\sum_{\ell=1}^{\lgQ*}\sum_{\bj\in \mathcal{J}_{\ell}}\sum_{\substack{\ba\in \mathbb{Z}^n}}\frac{w_{\bj}^*\aj}{\sqrt{\left|\det\, H_{F_{\bj}}\left((\nabla F_{\bj})^{-1}(\ba/j_1)\right)\right|}}\ll \frac{1}{D}\left(\sum_{\ell=1}^{\lgQ*} 2^{\ell(R-1)} \sum_{j_1=2^{\ell-1}}^{2^{\ell}}j_1^{n}\right)
    \ll (\delta^*)(Q^*)^{n+R},
\end{equation}
where the implicit constants depends only on $w, \cM$ and $\mathscr{L}$.
Next, to handle the terms with $d\neq 0$, we apply the $n$-dimensional Poisson summation formula and a change of variables to get
\begin{align}
    \label{eq dual nzero freq}
    &\sum_{\ba\in \mathbb{Z}^n}\frac{w_{\bj}^*\aj e\left(dj_1F_{\bj}^*\aj\right)}{\sqrt{\left|\det\, H_{F_{\bj}}\left((\nabla F_{\bj})^{-1}(\ba/j_1)\right)\right|}}\nonumber\\
    &=\sum_{\bk\in \mathbb{Z}^n}\int_{\mathbb{R}^n}\frac{w_{\bj}^*\left(\frac{\bz}{j_1}\right)}{\sqrt{\left|\det\, H_{F_{\bj}}\left((\nabla F_{\bj})^{-1}(\bz/j_1)\right)\right|}} e\left(dj_1F_{\bj}^*\left(\frac{\bz}{j_1}\right)-\bk\cdot\bz\right)\, d\bz=j_1^n\sum_{\bk\in \mathbb{Z}^n}I^*(d, \bj, \bk),
\end{align}
with $I^*(d, \bj, \bk)$ as in \eqref{def dual osc int}. From \eqref{eq dual N expsum}, \eqref{eq dual zero freq} and \eqref{eq dual nzero freq}, we conclude 
\begin{align*}
    {\mathfrak{N}}^{*, 1}_{w, \cM}(Q^*, \delta^*)&\ll(\delta^*)(Q^*)^{n+R}+\left|\sum_{\ell=1}^{\lgQ*}\sum_{\bj\in \mathcal{J}_{\ell}}\sum_{|d|=1}^D \frac{D-d}{D^2}\sum_{\substack{\bk\in \mathbb{Z}^n}}j_1^nI^*(d, \bj, \bk)\right|\\
    &\ll(\delta^*)(Q^*)^{n+R}+2\left|\sum_{\ell=1}^{\lgQ*}\sum_{\bj\in \mathcal{J}_{\ell}}\sum_{d=1}^D \frac{D-d}{D^2}\sum_{\substack{\bk\in \mathbb{Z}^n}}j_1^nI^*(d, \bj, \bk)\right|,
\end{align*}
where the last inequality follows from complex conjugation. 
\end{proof}

Let 
\begin{equation}
    \label{def dual phase}
    \varphi^{d, \bj, \bk}(\bx):=F_{\bj}^*(\bx)-\frac{\bk}{d}\cdot\bx,\qquad \varphi^{d, \bj, \bk}_1(\bx):=\frac{F_{\bj}^*(\bx)-\bk\cdot\bx}{\textrm{dist}(\bk, dU)}.
\end{equation}
Recall the index set $\mathcal{J}$ (defined in \eqref{eq J index set}), and the fact that $\nabla F_{\bj}$ is a diffeomorphism on the closure of $\mathscr{D}$ for all $\bj\in\mathcal{J}$. Further, recall the definitions of the sets $\mathscr{R}_{\bj}$ and $V_{\bj}$ from \eqref{eq def dual domains}.

Since $\nabla F^*_{\bj}$ is a diffeomorphism on $\mathscr{R}_{\bj}=\left(\nabla F_{\bj}\right)^{-1}(\mathscr{D})$, each $\bk\in d \mathscr{D}$ determines a unique critical point of $\varphi^{d,\bj, \bk}$ given by
\begin{equation}
    \label{def xdjk}
    x_{d, \bj, \bk}:=\left(\nabla F^*_{\bj}\right)^{-1}(\bk/d)=\nabla F_{\bj}(\bk/d)\in \mathscr{R}_{\bj}.
\end{equation}

To analyze the oscillatory integrals $I^*(d, \bj, \bk)$, we need a few preliminary estimates for the phase functions $\varphi^{d, \bj, \bk}$ (and $\varphi^{d, \bj, \bk}_1$) and the corresponding amplitudes. These are very similar to those proven in \cite[\S5 and \S6]{schindler2022density}. The main work is needed to show that these are independent of the parameters $d, \bj$ and $\bk$. We mention these estimates in the next lemma and briefly sktch their proofs. The interested reader is encouraged to consult \cite{schindler2022density} for the details.

\begin{lemma}
    \label{lem dual phase amp est}
    Let $\mathcal{J}(J_1)$ be as defined in \eqref{eq J index set}. Let $\balp\in \mathbb{Z}_{\geq 0}^n$ be an $n$-dimensional multi-index, $d\leq D$ be a positive integer, $\bj\in \mathcal{J}(J_1)$ and $\bk\in \mathbb{Z}^n$. Then
    
    (i) \begin{equation}
        \label{eq est dual phi1 deriv}
        \sup_{\bx\in V_{\bj}}\left|\frac{\partial^{\balp}\varphi_1^{d, \bj, \bk}}{\partial \bx^{\balp}}(\bx)\right|\ll 1.
    \end{equation}

    (ii)\begin{equation}
        \label{eq est dual phi deriv}
        \sup_{\bx\in V_{\bj}}\left|\frac{\partial^{\balp}\varphi^{d, \bj, \bk}}{\partial \bx^{\balp}}(\bx)\right|\ll 1.
    \end{equation}

(iii) \begin{equation}
        \label{eq est dual w deriv}
        \sup_{\bx\in V_{\bj}}\left|\frac{\partial^{\balp}}{\partial \bx^{\balp}}\left(w_{\bj}^*(|\det\, H_{F_{\bj}}|)^{-\frac{1}{2}}\right)(\bx)\right|\ll 1.
    \end{equation}

(iv) Under the additional assumption that $\bk\in d\mathscr{D}$, we have
\begin{equation}
    \label{eq est ibp dual}
    \sup_{\bx\in \mathscr{R}_{\bj}\setminus \{x_{d, \bj, \bk}\}}\frac{|\bx-\bx_{d, \bj, \bk}|}{|\nabla \varphi^{d, \bj, \bk}(\bx)|}\ll 1,
\end{equation}
where $\bx_{d, \bj, \bk}$ is as in \eqref{def xdjk}

The implicit constants in the estimates above depend on $\balp$ (for the first three inequalities), upper bounds for finitely many derivatives of $f_r$ ($1\leq r\leq R$) as well as $w$ on $U$, and the constant $\mathfrak{C}_0$ in the condition \eqref{eq curv est}, but are independent of $d, \bj$ and $\bk$.
\end{lemma}

\begin{proof}
    For the proof of (i), we refer the reader to Lemma 6.1 in \cite{schindler2022density}. The proof of (ii) also proceeds in the same way. 
    
    For part (iii), we also refer the reader to \cite{schindler2022density}. There is a slight difference though. In \cite{schindler2022density}, the authors use Lemma 6.1 (or its proof) to deduce uniform upper bounds for the derivatives of the weight $w_{\bj}^*$, whereas we need to establish \eqref{eq est dual w deriv} for $w_{\bj}^*|\det\, H_{F_{\bj}}|^{-\frac{1}{2}}$. However, this change is harmless owing to the condition \eqref{eq curv est} on the determinant of $H_{F_{\bj}}$. This is easily seen for the zeroth derivative. Using the Leibniz rule, we know that the higher order derivatives of $w_{\bj}^*|\det\, H_{F_{\bj}}|^{-\frac{1}{2}}$ are sums of products of derivatives of $w_{\bj}^*$ and $|\det\, H_{F_{\bj}}|^{-\frac{1}{2}}$. We note that any partial derivative of $|\det\, H_{F_{\bj}}|^{-\frac{1}{2}}$ is a real polynomial expression with uniformly bounded coefficients (independent of $\bj$ and $\bk$), in terms of the powers of $|\det\, H_{F_{\bj}}|^{-\frac{1}{2}}$ and the derivatives of $\nabla F_{\bj}$ (which are again bounded by constants independent of $\bj$ and $\bk$). Combining these observations with the proof of Lemma 6.1 in \cite{schindler2022density} establishes \eqref{eq est dual w deriv}.

    Finally, part (iv) is exactly the same as Lemma 6.2 in \cite{schindler2022density}.
\end{proof}

As is standard for these types of problems, we divide our consideration into three regimes, based on whether we need to apply the method of stationary phase, non-stationary phase or a hybrid argument. Let
\begin{equation}
    \label{def dual rho}
    \rho^*:=\frac{1}{2}\textrm{dist}(\partial \mathscr{D}, \partial U).
\end{equation}
We split $\mathbb{Z}^n=\mathscr{K}_1\cup \mathscr{K}_2\cup \mathscr{K}_3$ where
\begin{align*}
\mathscr{K}_1&:=\left\{\bk\in \mathbb{Z}^n: \frac{\bk}{d}\in U\right\},\\
\mathscr{K}_2&:=\left\{\bk\in \mathbb{Z}^n: \textrm{dist}\left(\frac{\bk}{d}, U\right)\geq \rho^*\right\},\text{ and}\\
\mathscr{K}_3&:=\mathbb{Z}^n\setminus\left(\mathscr{K}_1\cup \mathscr{K}_2\right).
\end{align*}
For $i\in \{1, 2, 3\}$, we define the respective contributions
\begin{equation}
    \label{def dual K1 contr}
    M_i^*(Q^*, \delta^*):=\left|\sum_{\ell=1}^{\lgQ*}\sum_{\bj\in \mathcal{J}_{\ell}}\sum_{d=1}^D \frac{D-d}{D^2}\sum_{\substack{\bk\in \mathscr{K}_i}}j_1^nI^*(d, \bj, \bk)\right|.
\end{equation}
By \eqref{eq lem dual main term}, we have
\begin{equation}
    \label{eq dualN diff regimes}
    {{N}}^{*, 1}_{w, \cM}(Q^*, \delta^*)\ll(\delta^*)(Q^*)^{n+R}+M_1^*(Q^*, \delta^*)+M_2^*(Q^*, \delta^*)+M_3^*(Q^*, \delta^*).
\end{equation}
We first estimate the contribution from the non-stationary regime $\mathscr{K}_2$. 
\begin{lemma}
    \label{lem dual nstation}
    $$M_2(Q^*, \delta^*)\ll (Q^*)^{R-1}\log (4Q^*),$$
with implicit constants depending only on $\cM$ and $w$.
\end{lemma}
\begin{proof}
Recall $\varphi^{d, \bj, \bk}$ from \eqref{def dual phase}, and let
$$\lambda_1=j_1\cdot \textrm{dist}(k, dU.)$$ For each $\bk\in \mathscr{K}_2$, we have the lower bound
$$\inf_{\bx\in V_{\bj}}|\nabla\varphi_1^{d, \bj, \bk}(\bx)|\geq 1.$$ 
Further, by parts (i) and (iii) of Lemma \ref{lem dual phase amp est}, we know that the derivatives of $\varphi_1^{d, \bj, \bk}$ and $w_{\bj}^*(\bx)|\det\, H_{F_{\bj}}|^{-\frac{1}{2}}$ are bounded on $V_{\bj}$ independently of $d, \bj$ and $\bk$. Thus we can apply Lemma \ref{lem non st phase} (integration by parts), with phase $\varphi^{d, \bj, \bk}_1$ and $\lambda=\lambda_1$, to conclude that 
$$I^*(d, \bj, \bk)\ll_m \lambda_1^{-m+1}=\left(j_1\cdot\textrm{dist}(\bk, dU)\right)^{-m+1},$$
for $m\in\mathbb{Z}_{\geq 0}$, with implicit constants independent of $d, \bj$ and $\bk$. In particular, taking $m\geq n+2$, we get
\begin{align*}
    \sum_{k\in \mathscr{K}_2} I^*(d, \bj, \bk)
    &\ll j_1^{-m+1}\sum_{\bk\in\mathscr{K}_2}\textrm{dist}\left(\bk, dU\right)^{-m+1}\\
    &\leq j_1^{-n-1} \sum_{i=0}^{\infty}\sum_{\substack{\bk\in \mathbb{Z}^n\\2^i\rho^*\leq \textrm{dist}(\bk, dU)<2^{(i+1)}\rho^*}}2^{-(i+1)(n+1)}(\rho^*)^{-(n+1)}\\
    &\ll j_1^{-n-1} \sum_{i=0}^{\infty} 2^{(i+1)(-n-1+n)}\leq j_1^{-n-1}.
\end{align*}
Thus 
\begin{align*}
    M_2^*(Q^*, \delta^*)\ll \sum_{\ell=1}^{\lgQ*}\sum_{\bj\in \mathcal{J}_{\ell}}\sum_{d=1}^D \frac{D-d}{D^2}j_1^{n-n-1}&\ll \left(\sum_{\ell=1}^{\lgQ*} 2^{\ell(R-1)} \sum_{j_1=2^{\ell-1}}^{2^{\ell}}j_1^{-1}\right)\\&\ll (Q^*)^{R-1}.
\end{align*}
\end{proof}

Next, we estimate the contributions from the intermediate and the stationary regimes. Recall $\varphi^{d,\bj, \bk}$ from \eqref{def dual phase} and $x_{d, \bj, \bk}$ from \eqref{def xdjk}.
The phase function $\varphi^{d,\bj, \bk}$ and the amplitude $\left(w_{\bj}^*(\bx)|\det\, H_{F_{\bj}}|^{-\frac{1}{2}}\right)$ satisfy the estimates $\eqref{eq est dual phi deriv}$ and \eqref{eq est dual w deriv} respectively. Further, 
\begin{equation}
    \label{eq dual hess est}
    H_{\varphi^{d,\bj, \bk}}\left(x_{d, \bj, \bk}\right)=H_{F^*_{\bj}}\left((\nabla F_{\bj}^*)^{-1}(\bk/d)\right)=H_{F_{\bj}}\left(\bk/d\right)\gg \mathfrak{C}_0^{-1},
\end{equation}
by \eqref{eq inv deriv} and \eqref{eq curv est}. 
We first consider the intermediate regime.
\begin{lemma}
    \label{lem dual interm}
    $$M_3(Q^*, \delta^*)\ll (Q^*)^{\frac{n}{2}+R-1}D^{\frac{n}{2}-1},$$
with implicit constants depending only on $\cM$ and $w$.
\end{lemma}
\begin{proof}
For $\bk\in \mathscr{K}_3\subseteq \mathscr{D}\setminus dU$, we know that $x_{d, \bj, \bk}\notin \nabla F_{\bj}(U)=\supp\, w_{\bj}^*.$ As discussed above, the phase function $\varphi^{d,\bj, \bk}$ and the amplitude $\left(w_{\bj}^*(\bx)|\det\, H_{F_{\bj}}|^{-\frac{1}{2}}\right)$ are well-behaved. We also have the lower bound \eqref{eq dual hess est}. 
Thus, we can apply Lemma \ref{lem st phase} (stationary phase principle), with phase $\varphi^{d, \bj, \bk}$ and $\lambda:=dj_1$, to conclude that
$$I^*(d, \bj, \bk)\ll (dj_1)^{-\frac{n}{2}-1}$$
for each $\bk\in \mathscr{K}_3$. 
Since $\#\mathscr{K}_3\ll_{U, \rho^*} d^n$, we can estimate
\begin{align*}
M_3(Q^*, \delta^*)&\ll\sum_{\ell=1}^{\lgQ*}\sum_{\bj\in \mathcal{J}_{\ell}}\sum_{d=1}^D \frac{D-d}{D^2}\sum_{\substack{\bk\in \mathscr{K}_3}}j_1^n |I^*(d, \bj, \bk)|\\
&\ll D^{-1}\sum_{\ell=1}^{\lgQ*} 2^{\ell(R-1)} \sum_{j_1=2^{\ell-1}}^{2^{\ell}}\sum_{d=1}^D\left(j_1 d\right)^{n-\frac{n}{2}-1} \\
&\ll (Q^*)^{\frac{n}{2}+R-1}D^{\frac{n}{2}-1}.
\end{align*}
\end{proof}

Finally, we estimate the contribution from the critical stationary phase regime.
\begin{lemma}
    \label{lem dual station}
    $$M_1(Q^*, \delta^*)\ll \frac{(Q^*)^{\frac{n}{2}}}{D}\sum_{d=1}^D d^{-\frac{n}{2}}\sum_{\bk\in \mathbb{Z}^n}{w\left(\frac{\bk}{d}\right)}\left(\prod_{r=1}^R\min\left(\|df_r(\bk/d)\|^{-1}, Q^*\right)\right)+(Q^*)^{\frac{n}{2}+R-1}D^{\frac{n}{2}-1},$$
with implicit constants depending only on $\cM$ and $w$.
\end{lemma}

\begin{proof}
We again intend to apply the stationary phase principle to evaluate the integrals $I^*(d, \bk, \bj)$ for $\bk\in dU$. As in the proof of Lemma \ref{lem dual interm}, we note that the phase $\varphi^{d,\bj, \bk}$ and the amplitude $\left(w_{\bj}^*(\bx)|\det\, H_{F_{\bj}}|^{-\frac{1}{2}}\right)$ satisfy the estimates $\eqref{eq est dual phi deriv}$ and \eqref{eq est dual w deriv} respectively.  
We also recall \eqref{eq dual hess est}. 
Further, since the eigenvalues of a matrix depend continuously on its coefficients, the condition \eqref{eq curv est} implies that the signature of $H_{\varphi^{d,\bj, \bk}}\left(x_{d, \bj, \bk}\right)$ is the same for all relevant values of $d, \bfj$ and $k$.
Let $\sigma$ denote this signature.

An application of Lemma \ref{lem st phase} (stationary phase principle), with $\lambda=dj_1$, phase $\varphi^{d, \bj, \bk}$ and amplitude function $\frac{w_{\bj}^*}{\sqrt{|\det\, H_{F_{\bj}^*}|}}$, gives
\begin{align*}
    I^{*}(d, \bj, \bk)&=\frac{w_{\bj}^*(\bx_{d, \bj, \bk})}{\sqrt{|\det\, H_{F_{\bj}^*}(\bx_{d, \bj, \bk})|}}\frac{1}{\sqrt{|\det H_{F_{\bj}}\left(\bk/d\right)|}}\left(j_1d\right)^{-\frac{n}{2}}e\left(-j_1d \varphi^{d, \bj, \bk}(\bx_{d, \bj, \bk})+\sigma/8\right)\nonumber\\
    &+O((j_1d)^{-\frac{n}{2}-1}).
\end{align*}
For $\bk\in \mathscr{K}_1$, we have
\begin{equation}
    \label{eq dual wthess est}
    \frac{w_{\bj}^*(\bx_{d, \bj, \bk})}{\sqrt{|\det\, H_{F_{\bj}^*}(\bx_{d, \bj, \bk})|}}={w\left(\frac{\bk}{d}\right)}{\sqrt{|\det H_{F_{\bj}}\left(\bk/d\right)|}}\,.
\end{equation}
We can also simplify
\begin{equation}
    \label{eq dual dual og}
    \varphi^{d, \bj, \bk}(\bx_{d, \bj, \bk})=\left(F^*_{\bj}\circ\left(\nabla F^*_{\bj}\right)^{-1}\right)\left(\frac{\bk}{d}\right)-\frac{\bk}{d}\cdot \left(\nabla F^*_{\bj}\right)^{-1}\left(\frac{\bk}{d}\right)=-F_{\bj}\left(\frac{\bk}{d}\right).
\end{equation}
Plugging \eqref{eq dual wthess est} and \eqref{eq dual dual og} into the stationary phase expansion for $I^{*}(d, \bj, \bk)$, we get
\begin{align}
\label{eq I* expsum}
    I^{*}(d, \bj, \bk)&=w\left(\frac{\bk}{d}\right)\left(j_1d\right)^{-\frac{n}{2}}e\left(-j_1d F_{\bj}\left(\frac{\bk}{d}\right)+\frac{\sigma}{8}\right)+O((j_1d)^{-\frac{n}{2}-1})\nonumber\\
    &=w\left(\frac{\bk}{d}\right)\left(j_1d\right)^{-\frac{n}{2}}e\left(-d \sum_{r=1}^Rj_rf_r\left(\frac{\bk}{d}\right)+\frac{\sigma}{8}\right)+O((j_1d)^{-\frac{n}{2}-1}).
\end{align}
Now
\begin{align*}
&\left|\sum_{\ell=1}^{\lgQ*}\sum_{\bj\in \mathcal{J}_{\ell}}j_1^{n-\frac{n}{2}}e\left(-d \sum_{r=1}^Rj_rf_r\left(\frac{\bk}{d}\right)\right)\right|\\
&\leq\sum_{\ell=1}^{\lgQ*}\left(\left|\sum_{j_1\in \mathbb{Z}\cap[2^{\ell-1}, 2^{\ell})}j_1^{\frac{n}{2}}e\left(-d j_1f_1\left(\frac{\bk}{d}\right)\right)\right|\right)\left(\prod_{r=2}^R \left|\sum_{j_r\in \mathbb{Z}\cap[0, 2^{\ell}]}e\left(-d j_rf_r\left(\frac{\bk}{d}\right)\right)\right|\right).
\end{align*}
By geometric summation, for $2\leq r\leq R$, we have
\begin{equation*}
    \left|\sum_{j_r\in \mathbb{Z}\cap[0, 2^{\ell}]}e\left(-d j_rf_r\left(\frac{\bk}{d}\right)\right)\right|\leq \min\left(2^\ell, \left\|df_r\left(\bk/d\right)\right\|^{-1}\right),
\end{equation*}
while for $r=1$, we use geometric and partial summation to conclude that
\begin{equation*}
    \left|\sum_{j_1\in \mathbb{Z}\cap[2^{\ell-1}, 2^{\ell}]}j_1^{\frac{n}{2}}e\left(-d j_1f_1\left(\bk/d\right)\right)\right|\leq 2^{\ell\frac{n}{2}}\min\left(2^\ell, \left\|df_1\left(\bk/d\right)\right\|^{-1}\right).
\end{equation*}
Putting the above estimates together, we conclude that
\begin{align}
\left|\sum_{\ell=1}^{\lgQ*}\sum_{\bj\in \mathcal{J}_{\ell}}j_1^{\frac{n}{2}}e\left(-d \sum_{r=1}^Rj_rf_r\left(\bk/d\right)\right)\right|
&\leq \sum_{\ell=1}^{\lgQ*}2^{\ell\frac{n}{2}} \prod_{r=1}^R \min\left(2^\ell, \left\|df_r\left(\bk/d\right)\right\|^{-1}\right)\nonumber\\
&\leq \left(\sum_{\ell=1}^{\lgQ*}2^{\ell\frac{n}{2}}\right)\prod_{r=1}^R \min\left(2Q^*, \left\|df_r\left(\bk/d\right)\right\|^{-1}\right)\nonumber\\  &\leq 4(Q^*)^{\frac{n}{2}} \prod_{r=1}^R  \min\left(Q^*, \left\|df_r\left(\bk/d\right)\right\|^{-1}\right).
\label{eq dual diff est}
\end{align}
We now have all ingredients in place to estimate $M_3^*(Q^*, \delta^*)$. We first use \eqref{eq I* expsum} to conclude that 
$M_3^*(Q^*, \delta^*)$ is dominated by
\begin{align}
    \label{eq K1 t1}
     &\sum_{d=1}^D \frac{D-d}{D^2} d^{-\frac{n}{2}}\sum_{\substack{\bk\in \mathscr{K}_3}}w\left(\frac{\bk}{d}\right)\left|\sum_{\ell=1}^{\lgQ*}\sum_{\bj\in \mathcal{J}_{\ell}}j_1^{n-\frac{n}{2}}e\left(-d \sum_{r=1}^Rj_rf_r\left(\frac{\bk}{d}\right)\right)\right|\\
      \label{eq K1 t2}
     &+\sum_{\ell=1}^{\lgQ*}\sum_{\bj\in \mathcal{J}_{\ell}}\sum_{d=1}^D \frac{D-d}{D^2}\sum_{\substack{\bk\in \mathscr{K}_3}}j_1^{n-\frac{n}{2}-1}d^{-\frac{n}{2}-1}.
\end{align}
Using \eqref{eq dual diff est}, we deduce that \eqref{eq K1 t1} can be bounded from above by a positive constant times
$$\frac{(Q^*)^{\frac{n}{2}}}{D}\sum_{d=1}^D d^{-\frac{n}{2}}\sum_{\bk\in \mathbb{Z}^n}{w\left(\frac{\bk}{d}\right)}\left(\prod_{r=1}^R\min\left(\|df_r(\bk/d)\|^{-1}, Q^*\right)\right).$$
On the other hand, using the fact that $(\#\mathscr{K}_3)\ll d^{n}$, \eqref{eq K1 t2} can be estimated as follows 
\begin{align*}
\sum_{\ell=1}^{\lgQ*}\sum_{\bj\in \mathcal{J}_{\ell}}\sum_{d=1}^D \frac{D-d}{D^2}&\sum_{\substack{\bk\in \mathscr{K}_3}}j_1^{n-\frac{n}{2}-1}d^{-\frac{n}{2}-1}\leq \frac{1}{D}\left(\sum_{\ell=1}^{\lgQ*} 2^{\ell(R-1)} \sum_{j_1=2^{\ell-1}}^{2^{\ell}}j_1^{\left(\frac{n}{2}-1\right)}\right)\left(\sum_{d=1}^D (\#\mathscr{K}_3)d^{-\frac{n}{2}-1}\right)\\   
&\ll\left(\sum_{\ell=1}^{\lgQ*} 2^{\ell\left(\frac{n}{2}+R-1\right)} \right)\left(D^{-1}\sum_{d=1}^D d^{n-\frac{n}{2}-1}\right)\leq (Q^*)^{\frac{n}{2}+R-1}D^{\frac{n}{2}-1}.
\end{align*}
Adding the two estimates finishes the proof.
\end{proof}
\subsubsection*{Concluding the proof of Proposition \ref{prop dual vdc b}}
Combining \eqref{eq dualN diff regimes} with Lemmas \ref{lem dual nstation}-\ref{lem dual station}, we conclude that
\begin{align*}
    {\mathfrak{N}}^{*, 1}_{w, \cM}(Q^*, \delta^*)\leq {{N}}^{*, 1}_{w, \cM}(Q^*, \delta^*)&\ll(\delta^*)(Q^*)^{n+R}+(Q^*)^{R-1}
    +(Q^*)^{\frac{n}{2}+R-1}D^{\frac{n}{2}-1}\\
    &+\frac{(Q^*)^{\frac{n}{2}}}{D}\sum_{d=1}^D d^{-\frac{n}{2}}\sum_{\bk\in \mathbb{Z}^n}{w\left(\frac{\bk}{d}\right)}\left(\prod_{r=1}^R\min\left(\|df_r(\bk/d)\|^{-1}, Q^*\right)\right),
\end{align*}
which implies \eqref{eq dual vdc b}.

\section{Proof of Proposition \ref{prop og vdc b}}
\label{sec og vdc b}
Recall that \begin{equation*}
    \label{eq Nog 11 sum2}
    {\mathfrak{N}}^1_{w, \cM}(Q, \bdel)= C_{w} \delp Q^{n+1}+ \sum_{\bj\in\mathcal{J}(J_1)} \left(\prod_{r=1}^R \frac{1}{j_r+1}\right)\left|\sum_{\substack{\ba\in \mathbb{Z}^n\\1\leq q\leq Q}} w\left(\frac{\ba}{q}\right)\exp \left(qj_1 F_{\bj}\aq\right)\right|,
\end{equation*}
where $\mathcal{J}$ is as defined in \eqref{eq J index set} with $s=1$. Using the $n$-dimensional Poisson summation formula (and a change of variables) for the sum inside absolute values, we can write
$$\sum_{\substack{\ba\in \mathbb{Z}^n\\1\leq q\leq Q}} w\left(\frac{\ba}{q}\right)\exp \left(qj_1 F_{\bj}\aq\right)=\sum_{\substack{\bk\in \mathbb{Z}^n\\1\leq q\leq Q}} q^n I (q, \bj, \bk),$$
with
\begin{equation}
 \label{eq Iqjk}   
 I (q, \bj, \bk)=\int_{\mathbb{R}^n} \exp\left(qj_1\left(F_{\bj}(\bx)- \frac{\bk\cdot\bx}{j_1}\right)\right) w(\bx)\, d\bx.
\end{equation}
We thus have
\begin{equation}
    \label{eq Nog 11 osc sum}
    {\mathfrak{N}}^1_{w, \cM}(Q, \bdel)= C_{w} \delp Q^{n+1}+ \sum_{\bj\in\mathcal{J}(J_1)} \left(\prod_{r=1}^R \frac{1}{j_r+1}\right)\left|\sum_{\substack{\bk\in \mathbb{Z}^n\\1\leq q\leq Q}} q^n I (q, \bj, \bk)\right|.
\end{equation}
Recall the set $\mathscr{D}$ from \eqref{eq def og domains}, and that $\nabla F_{\bj}$ is a diffeomorphism on $\mathscr{D}$ for all $\bj\in\mathcal{J}$. Further, recall the definitions of the sets $\mathscr{R}_{\bj}$ and $V_{\bj}$ from \eqref{eq def dual domains} (for $s=1$).

Let 
\begin{equation}
    \label{def og phase}
    \varphi^{\bj, \bk}(\bx):=F_{\bj}(\bx)-\frac{\bk}{j_1}\cdot\bx,\qquad \varphi^{\bj, \bk}_1(\bx):=\frac{F_{\bj}(\bx)-\bk\cdot\bx}{\textrm{dist}(\bk, j_1 V_{\bj})}.
\end{equation}
Since $\nabla F_{\bj}$ is a diffeomorphism mapping $\mathscr{D}$ onto $\mathscr{R}_{\bj}$, each $\bk\in j_1 \mathscr{R}_{\bj}$ determines a unique critical point of $\varphi^{\bj, \bk}$ given by
\begin{equation}
    \label{def xjk}
    x_{\bj, \bk}:=\left(\nabla F_{\bj}\right)^{-1}(\bk/j_1)\in \mathscr{D}.
\end{equation}
As in the last section, to analyze the oscillatory integrals $I(d, \bj, \bk)$, we need a few preliminary estimates for the phase functions $\varphi^{\bj, \bk}$ (and $\varphi^{\bj, \bk}_1$) and the corresponding amplitude functions. These estimates are again very similar to those proven in \cite[\S5]{schindler2022density}, with the main work spent in showing that they are independent of the parameters $\bj$ and $\bk$. The following lemma is analogous to Lemma \ref{lem dual phase amp est} from the previous section.
\begin{lemma}
    \label{lem og phase amp est}
    Let $\mathcal{J}(J_1)$ be as defined in \eqref{eq J index set}. Let $\balp\in \mathbb{Z}_{\geq 0}^n$ be an $n$-dimensional multi-index,  $\bj\in \mathcal{J}(J_1)$ and $\bk\in \mathbb{Z}^n$. Then
    
    (i) \begin{equation}
        \label{eq est og phi1 deriv}
        \sup_{\bx\in U}\left|\frac{\partial^{\balp}\varphi_1^{\bj, \bk}}{\partial \bx^{\balp}}(\bx)\right|\ll 1.
    \end{equation}

    (ii)\begin{equation}
        \label{eq est og phi deriv}
        \sup_{\bx\in U}\left|\frac{\partial^{\balp}\varphi^{\bj, \bk}}{\partial \bx^{\balp}}(\bx)\right|\ll 1.
    \end{equation}

(iii) Under the additional assumption that $\bk\in j_1\mathscr{D}$, we have
\begin{equation}
    \label{eq est ibp og}
    \sup_{\bx\in \mathscr{D}\setminus \{x_{\bj, \bk}\}}\frac{|\bx-\bx_{\bj, \bk}|}{|\nabla \varphi^{\bj, \bk}(\bx)|}\ll 1,
\end{equation}
where $\bx_{\bj, \bk}$ is as in \eqref{def xjk}.

The implicit constants in the estimates above depend on $\balp$ (for the first three inequalities), upper bounds for finitely many derivatives of $f_r$ ($1\leq r\leq R$) as well as $w$ on $U$, and the constant $\mathfrak{C}_0$ in the condition \eqref{eq curv est}, but are independent of $\bj$ and $\bk$.
\end{lemma}

\begin{proof}
    See Lemmas 5.1 and 5.2 in \cite{schindler2022density}. 
\end{proof}

We divide our consideration into the stationary, non-stationary and intermediate regimes. Recall the sets $\mathscr{D}=B_{2\varepsilon_0}(\bx_0)$ and $U:=\supp\, w\subseteq B_{\varepsilon_0}(\bx_0)$ from \eqref{eq def og domains}. Let
\begin{eqnarray}
\notag
\rho:=\frac{1}{2}\inf_{\mathbf{t}\in [0, 2]^{R-1}}\textrm{dist} \left( \partial  \left( \nabla_{\bx} ( f_1  + \sum_{r=2}^Rt_r f_r ) (\mathscr{D}) \right), \partial  \left( \nabla ( f_1  + \sum_{r=2}^R t_r f_r ) (B_{\varepsilon_0}(\bx_0)) \right)  \right)>0.
\\
\label{def og rho}
\end{eqnarray}
For each $\bj\in\mathcal{J}(J_1)$, we split $\mathbb{Z}^n=\mathscr{K}_{\bj, 1}\cup \mathscr{K}_{\bj, 2}\cup \mathscr{K}_{\bj, 3}$ where
\begin{align*}
\mathscr{K}_{\bj, 1}&:=\left\{\bk\in \mathbb{Z}^n: \frac{\bk}{j_1}\in V_{\bj}\right\},\\
\mathscr{K}_{\bj, 2}&:=\left\{\bk\in \mathbb{Z}^n: \textrm{dist}\left(\frac{\bk}{j_1}, V_{\bj}\right)\geq \rho\right\},\text{ and}\\
\mathscr{K}_{\bj, 3}&:=\mathbb{Z}^n\setminus\left(\mathscr{K}_{\bj, 1}\cup \mathscr{K}_{\bj, 2}\right).
\end{align*}
For $i\in \{1, 2, 3\}$, we define the respective contributions
\begin{equation}
    \label{def og K1 contr}
    M_i(Q, \bdel):=\sum_{\bj\in\mathcal{J}(J_1)} \left(\prod_{r=1}^R \frac{1}{j_r+1}\right)\left|\sum_{\substack{\bk\in \mathscr{K}_{\bj, i}\\1\leq q\leq Q}} q^n I (q, \bj, \bk)\right|.
\end{equation}
By \eqref{eq Nog 11 osc sum}, we then have
\begin{equation}
    \label{eq ogN diff regimes}
    {\mathfrak{N}}^{1}_{w, \cM}(Q, \bdel)\leq C_{w} \delp Q^{n+1}+M_1(Q, \bdel)+M_2(Q, \bdel)+M_3(Q, \bdel).
\end{equation}
We first estimate the contribution from the non-stationary regime $\mathscr{K}_{\bj, 2}$.
\begin{lemma}
    \label{lem og nstation}
    $$M_2(Q, \bdel)\ll (\log 4Q) (\log 4J_1)^R,$$
with the implicit constant depending only on $\cM$ and $w$.
\end{lemma}
\begin{proof}
Recall $\varphi^{\bj, \bk}$ from \eqref{def og phase}, and let
$$\lambda_1=q\cdot \textrm{dist}(k, j_1V_{\bj}.)$$ For each $\bk\in \mathscr{K}_{\bj, 2}$, we have the lower bound
$$\inf_{\bx\in U}|\nabla\varphi_1^{\bj, \bk}(\bx)|\geq 1.$$ 
Further, by part (i) of Lemma \ref{lem og phase amp est}, we know that the derivatives of $\varphi_1^{\bj, \bk}$ are bounded on $U$ independently of $\bj$ and $\bk$. Thus we can apply Lemma \ref{lem non st phase} (integration by parts), with phase $\varphi^{\bj, \bk}_1$ and $\lambda=\lambda_1$, to conclude that 
$$I(q, \bj, \bk)\ll_m \lambda_1^{-m+1}=\left(q\cdot\textrm{dist}(\bk, j_1V_{\bj})\right)^{-m+1},$$
for $m\in\mathbb{Z}_{\geq 0}$, with implicit constants independent of $q, \bj$ and $\bk$. In particular, taking $m\geq n+2$ and arguing like in the proof of Lemma \ref{lem dual nstation}, we get
\begin{align*}
    \sum_{k\in \mathscr{K}_{\bj, 2}} I(q, \bj, \bk)
    \ll q^{-n-1} \sum_{i=0}^{\infty} 2^{(i+1)(-n-1+n)}\leq q^{-n-1}.
\end{align*}
Thus 
\begin{align*}
     M_2(Q, \bdel)\ll\sum_{\bj\in\mathcal{J}(J_1)} \left(\prod_{r=1}^R      \frac{1}{j_r+1}\right)\left|\sum_{\substack{\bk\in \mathscr{K}_{\bj, 2}\\1\leq q\leq Q}} q^n I (q, \bj, \bk)\right|&\ll   \sum_{j_1=1}^{J_1}\sum_{j_2, \ldots, j_R=0}^{j_1} \left(\prod_{r=1}^R      \frac{1}{j_r+1}\right)\sum_{q=1}^Q q^{n-n-1} \\
     &\ll (\log 4Q) (\log4J_1)^R.
\end{align*}
\end{proof}


Next, we estimate the contributions from the intermediate and the stationary regimes. Recall $\varphi^{\bj, \bk}$ from \eqref{def og phase} and $x_{\bj, \bk}$ from \eqref{def xjk}. Recall that that the phase $\varphi^{\bj, \bk}$ satisfies the estimates $\eqref{eq est og phi deriv}$. Further, 
\begin{equation}
    \label{eq og hess est}
    H_{\varphi^{\bj, \bk}}\left(x_{\bj, \bk}\right)=H_{F_{\bj}}\left((\nabla F_{\bj})^{-1}(\bk/j_1)\right)
\gg \mathfrak{C}_0^{-1},
\end{equation}
by \eqref{eq inv deriv} and \eqref{eq curv est}. We again consider the intermediate regime first.
\begin{lemma}
    \label{lem og interm}
    $$M_3(Q, \bdel)\ll Q^{\frac{n}{2}}J_1^{\frac{n}{2}-1}(\log 4J_1)^R,$$
with the implicit constant depending only on $\cM$ and $w$.
\end{lemma}
\begin{proof}
For $\bk\in \mathscr{K}_{\bj, 3}\subseteq \mathscr{R}_{\bj}\setminus j_1V_{\bj}$, we know that $x_{\bj, \bk}\notin \left(\nabla F_{\bj}\right)^{-1}(V_{\bj})=U\supset \supp\, w$. As discussed above, the phase function $\varphi^{\bj, \bk}$ is well-behaved. We also have the lower bound \eqref{eq og hess est}. Thus, we can apply Lemma \ref{lem non st phase} (stationary phase principle) with phase $\varphi^{\bj, \bk}$ and $\lambda:=qj_1$, to conclude that
$$I(q, \bj, \bk)\ll (qj_1)^{-\frac{n}{2}-1}$$
for each $\bk\in \mathscr{K}_{\bj, 3}$.
Since $\#\mathscr{K}_3\ll_{U, \rho} j_1^n$, we can estimate
\begin{align*}
 M_3(Q, \bdel)\ll \sum_{\bj\in\mathcal{J}(J_1)} \left(\prod_{r=1}^R \frac{1}{j_r+1}\right)j_1^{n-\frac{n}{2}-1}\sum_{1\leq q\leq Q} q^{n-\frac{n}{2}-1}&\leq \left(\sum_{\bj\in\mathcal{J}(J_1)}\prod_{r=1}^R  \frac{1}{j_r+1}\right)Q^{\frac{n}{2}}J_1^{\frac{n}{2}-1}\\
 &\leq Q^{\frac{n}{2}}J_1^{\frac{n}{2}-1}(\log 4J_1)^{R}. 
\end{align*}
\end{proof}

Finally, we estimate the contribution from the critical stationary phase regime.
\begin{lemma}
    \label{lem og station}
    \begin{align*}
    M_1(Q, \bdel)&\ll Q^{\frac{n}{2}}  \sum_{\bj\in \mathcal{J}(J_1)}
    \left(\prod_{r=1}^R \frac{1}{j_r+1}\right)\sum_{\bk\in 
    \mathbb{Z}^n}\frac{w_{\bj}^*\left(\frac{\bk}{j_1}\right)}{\sqrt{|\det\, H_{F_{\bj}}\left((\nabla F_{\bj})^{-1}(\bk/j_1)\right)|}}j_1^{-\frac{n}{2}}\min\left(\|j_1F^*_{\bj}(\bk/j_1)\|^{-1}, Q\right)\nonumber\\
    &+Q^{\frac{n}{2}}J_1^{\frac{n}{2}-1}\left(\log 4J_1\right)^R,
    \end{align*}     
with implicit constants depending only on $\cM$ and $w$.
\end{lemma}
\begin{proof}
We intend to apply the stationary phase principle to evaluate the integrals $I(q, \bk, \bj)$ for $\bk\in j_1V_{\bj}$. As in the proof of Lemma \ref{lem og interm}, we note that the phase $\varphi^{\bj, \bk}$ satisfies the estimates $\eqref{eq est og phi deriv}$. We also recall \eqref{eq og hess est}. 
Further, the condition \eqref{eq curv est} implies that the signature of $H_{\varphi^{\bj, \bk}}\left(x_{\bj, \bk}\right)$ is the same for all relevant values of $\bfj$ and $k$.
Let $\sigma$ denote this signature.
An application of the stationary phase principle (Lemma \ref{lem st phase}) with $\lambda=dj_1$, phase $\varphi^{\bj, \bk}$ and amplitude function $w$, gives
\begin{equation*}
    I(q, \bj, \bk)=\frac{w(\bx_{\bj, \bk})}{\sqrt{|\det\, H_{F_{\bj}}\left((\nabla F_{\bj})^{-1}(\bk/j_1)\right)|}}\left(qj_1\right)^{-\frac{n}{2}}e\left(-qj_1 \varphi^{\bj, \bk}(\bx_{\bj, \bk})+\sigma/8\right)+O((qj_1)^{-\frac{n}{2}-1}).
\end{equation*}
For $\bk\in \mathscr{K}_{\bj, 1}$, we have
\begin{equation}
    \label{eq og wt est}
    {w(\bx_{\bj, \bk})}=w\left((\nabla F_{\bj})^{-1}(\bk/j_1)\right)=w^*_{\bj}\left(\frac{\bk}{j_1}\right).
\end{equation}
We can also simplify
\begin{equation}
    \label{eq og og dual}
    \varphi^{\bj, \bk}(\bx_{\bj, \bk})=\left(F_{\bj}\circ\left(\nabla F_{\bj}\right)^{-1}\right)\left(\frac{\bk}{j_1}\right)-\frac{\bk}{j_1}\cdot \left(\nabla F_{\bj}\right)^{-1}\left(\frac{\bk}{j_1}\right)=-F^{*}_{\bj}\left(\frac{\bk}{j_1}\right).
\end{equation}
Plugging \eqref{eq og wt est} and \eqref{eq og og dual} into the stationary phase expansion for $I(q, \bj, \bk)$, we get
\begin{align}
\label{eq I expsum}
    I(q, \bj, \bk)=\frac{w^*\left(\frac{\bk}{j_1}\right)}{\sqrt{|\det\, H_{F_{\bj}}\left((\nabla F_{\bj})^{-1}(\bk/j_1)\right)|}}\left(qj_1\right)^{-\frac{n}{2}}e\left(-qj_1 F^{*}_{\bj}\left(\frac{\bk}{j_1}\right)+\frac{\sigma}{8}\right)+O((qj_1)^{-\frac{n}{2}-1}).
\end{align}
By geometric and partial summation, we have
\begin{align*}
    \sum_{q=1}^{Q} q^n I(q, \bj, \bk)\ll Q^{\frac{n}{2}}\frac{w^*\left(\frac{\bk}{j_1}\right)}{\sqrt{|\det\, H_{F_{\bj}}\left((\nabla F_{\bj})^{-1}(\bk/j_1)\right)|}}j_1^{-\frac{n}{2}}\min \left(\|j_1F^*_{\bj}(\bk/j_1)\|^{-1}, Q\right)+Q^{\frac{n}{2}}j_1^{-\frac{n}{2}-1}.
\end{align*}
Therefore
\begin{align}
    \label{lem og st pre sec term}
    M_1(Q, \bdel)&\ll Q^{\frac{n}{2}}
    \sum_{\bj\in \mathcal{J}(J_1)}\left(\prod_{r=1}^R \frac{1}{j_r+1}\right)\sum_{\bk\in \mathbb{Z}^n}\frac{w_{\bj}^*\left(\frac{\bk}{j_1}\right)}{\sqrt{|\det\, H_{F_{\bj}}\left((\nabla F_{\bj})^{-1}(\bk/j_1)\right)|}}j_1^{-\frac{n}{2}}\min\left(\|j_1F^*_{\bj}(\bk/j_1)\|^{-1}, Q\right)\nonumber\\
    &+
    \sum_{\bj\in\mathcal{J}(J_1)}\left(\prod_{r=1}^R \frac{1}{j_r+1}\right)\sum_{\bk\in \mathscr{K}_{\bj, 1}}Q^{\frac{n}{2}}j_1^{-\frac{n}{2}-1}.
\end{align}
Arguing as in the proof of Lemma \ref{lem og interm}, we can deduce that $\#\mathscr{K}_{\bj, 1}\ll j_1^{n}.$ Thus the second term in \eqref{lem og st pre sec term} can be estimated as below
\begin{align*}
    \sum_{\bj\in\mathcal{J}(J_1)}\left(\prod_{r=1}^R \frac{1}{j_r+1}\right)\sum_{\bk\in \mathscr{K}_{\bj, 1}}Q^{\frac{n}{2}}j_1^{-\frac{n}{2}-1}\ll Q^{\frac{n}{2}}
    \sum_{j_1=1}^{J_1}\sum_{j_2, \ldots, j_R=0}^{j_1}\left(\prod_{r=1}^R \frac{1}{j_r+1}\right) j_1^{\frac{n}{2}-1}\leq Q^{\frac{n}{2}}J_1^{\frac{n}{2}-1}(\log 4J_1)^R.
\end{align*}
This establishes Lemma \ref{lem og station}.
\end{proof}
\subsubsection*{Concluding the proof of Proposition \ref{prop og vdc b}} The required estimate follows by combining \eqref{eq ogN diff regimes} with Lemmas \ref{lem og nstation}-\ref{lem og station}.

\section{Proof of Theorem \ref{thm homog main}}
\label{sec thm homog main}
Let
\begin{equation}
    \label{def N0}
    \mathfrak{N}_{0}(Q):=\sum_{\substack{\ba\in \mathbb{Z}^n\\1\leq q\leq Q}}w\left(\frac{\ba}{q}\right).
\end{equation}
Applying the Poisson summation formula and using the rapid decay of $\hat{w}$, we get
\begin{equation}
    \label{eq No expr}
    \mathfrak{N}_0=\sum_{1\leq q\leq Q} q^{n}\sum_{\bk\in \mathbb{Z}^n}\hat{w}(q \bk) = \frac{\hat{w}(\bzero)}{n+1}Q^{n+1}+ O(Q^{n}).
\end{equation}

Next, we use an $R$-fold product of the Selberg magic functions of degree $X$ to estimate the characteristic function of the set $(-\delta, \delta)$. Let
$$b_{j_r}:=\frac{1}{X+1}+\min \left\{2\delta, \frac{1}{j_r}\right\},$$
and
\begin{equation}
    \label{def Er}
    E_{s, \bgam} (Q, X):=\sum_{\substack{\bj\in \mathbb{Z}^R_{\geq 0}:\\\|\bj\|_{\infty}=j_s\in [1, X]}} \left(\prod_{r=1}^R b_{j_r}\right)\left|\sum_{\substack{\ba\in \mathbb{Z}^n\\1\leq q\leq Q}} w\left(\frac{\ba}{q}\right)\exp \left(\sum_{r=1}^R\gamma_rj_r q f_r \aq\right)\right|.
\end{equation}
Separating the term corresponding to $j_1=\ldots=j_R=0$ and using \eqref{eq selb0}, we can write
\begin{align}
    \label{eq Selb exp}
    |\Nwmh-2 \delta^R \mathfrak{N}_0| &\ll \delta^{R-1}\frac{1}{X}Q^{n+1}+\frac{1}{X^R} Q^{n+1} +\sum_{s=1}^R \sum_{\bgam\in \{-1, 1\}^R}E_{s, \bgam} (Q, X), 
\end{align}

For ease of exposition, let us focus on the case when $s=1$ and $\bgam=(1, 1, \ldots, 1)\in \mathbb{Z}^R$; the other cases can be dealt with using the same argument. We shall suppress notation and refer to $E_{1, (1, \ldots, 1)}(Q, X)$ simply as $E(Q, X)$. Arguing the same way as in the proof of Proposition \ref{prop og vdc b}, we can conclude that
\begin{align}
        E(Q, X)&\ll Q^{\frac{n}{2}}\sum_{\substack{1\leq j_1\leq X\\ 0\leq j_{r'}\leq j_1\\ 2\leq r'\leq R}}\left(\prod_{r=1}^R b_{j_r}\right)\sum_{\bk\in \mathbb{Z}^n}\frac{w_{\bj}^*\kj}{\sqrt{|\det\, H_{F_{\bj}}\left((\nabla F_{\bj})^{-1}(\bk/j_1)\right)|}}j_1^{-\frac{n}{2}}\min\left(\|j_1F^*_{\bj}(\bk/j_1)\|^{-1}, Q\right)\nonumber\\
       &+X^{\frac{n}{2}-1}Q^{\frac{n}{2}}\left(\log 4X\right)^R.
        \label{eq EQX}
    \end{align}
with the implicit constant depending only on $w$ and $\mathcal{M}$. We recall that the function $F_{\bj}$ is given by
$$F_{\bj}(\bx):=f_1(\bx)+\sum_{r=2}^R \frac{j_r}{j_1}f_r(\bx).$$
We will need two preliminary lemmas. The first one is a consequence of Theorem \ref{thm main}. 
   
\begin{lemma}
\label{lem 1 lower bd}
Let $\alpha_{\textrm{st}}$ be as defined in \eqref{def alpha stop}. There exists a positive constant $C_3$ depending only on $w$ and $\cM$ such that for all $Q^*\geq 1$ and for all $\delta^*\in (0, 1/2)$, we have 
\begin{equation}
    \label{eq thm1lem1 concl}
    \sum_{\substack{1\leq j_1\leq Q^*\\ 0\leq j_{r'}\leq j_1\\ 2\leq r'\leq R}}\sum_{\substack{\bk\in \mathbb{Z}^n\\\|j_1 F^*_{\bj}(\bk/j_1)\|<\delta^{*}}}\frac{w_{\bj}^*\kj}{\sqrt{|\det\, H_{F_{\bj}}\left((\nabla F_{\bj})^{-1}(\bk/j_1)\right)|}} \leq C_3\delta^* (Q^*)^{n+R} +C_3\left(Q^*\right)^{\alpha_{\textrm{st}}}\mathcal{E}_n\left(Q^*\right),
\end{equation}
where $\mathcal{E}_n(Q)$ is as defined in \eqref{eq aux error}.
\end{lemma}
\begin{proof}
The proof proceeds in exactly the same way as that of
Proposition \ref{prop dual from og}, except that we use the conclusion of Theorem \ref{thm main} instead of \eqref{eq p2 hypo} as our input. Indeed, by Theorem \ref{thm main}, we know that for all $Q\geq 1$ and for all $\bdel\in (0, 1/2)^R$, 
\begin{equation*}
   \mathfrak{N}_{w, \cM}(Q, \delta)\ll 
   \mathcal{E}_n(Q)\left(\delta^R Q^{n+1}+ \sum_{r=1}^R \delta^{R-1} Q^{n+1+\frac{\Theta-(n+1)}{R}}+Q^{\Theta}\right),
\end{equation*}
with $\Theta$ is as defined in \eqref{def theta}.
Then by following the proof of Proposition \ref{prop dual from og}, we can conclude that there exists a positive constant $C_3$ depending only on $w$ and $\cM$, such that for all $Q^*\geq 1$ and for all $\delta^*\in (0, 1/2)$, we have 
\begin{equation*}
    \DNwm \leq C_3\delta^* (Q^*)^{n+R} +C_3\left(Q^*\right)^{\alpha_{\textrm{st}}}\mathcal{E}_n\left(Q^*\right),
\end{equation*}
with 
\begin{equation*}
    \alpha_{\textrm{st}}=n+R-\frac{n }{2 \Theta- n}= \max\left(\frac{n(n+R+1)}{n+2}, n+R-1-\frac{2}{n}\right).
\end{equation*}
Here $\DNwm$ is as defined in \eqref{eq def dual weight Hessian} and is exactly the left hand side of \eqref{eq thm1lem1 concl}. Thus we are done.
\end{proof}

The second lemma takes the conclusion of the previous one as input to derive an estimate for $E(Q, X)$, using \eqref{eq EQX}. The proof is exactly the same as the first half of the proof of Proposition \ref{prop og from dual}, using Proposition \ref{prop og vdc b}.
\begin{lemma}
\label{lem 2 low bd}
There exists a positive constant $C_4$ depending only on $w$ and $\cM$ such that for all $Q\geq 1$ and for all $X\in (2, \infty)$, we have 
\begin{equation*}
    E(Q, X)\leq C_4 \mathcal{E}_n(X)\left(\log 4X\right)^{R}\left( (\log 4Q)Q^{\frac{n}{2}}X^{\frac{n}{2}}+Q^{\frac{n}{2}+1}X^{\alst-\frac{n}{2}-R}\right).
\end{equation*}
\end{lemma}    
\begin{proof}
    We proceed as in the proof of Proposition \ref{prop og from dual}. We make a dyadic decomposition as in \eqref{eq p3 1} based on the size of $\|j_1F^*_{\bj}(\bk/j_1)\|$ and then use Lemma \ref{lem 1 lower bd} and partial summation on the $j_r$ variables, to conclude that for each $0\leq i\leq \frac{\log 4Q}{\log 2}$, we have
    \begin{align*}
   &\sum_{\substack{1\leq j_1\leq X\\ 0\leq j_r\leq j_1\\ 2\leq r\leq R}}\left(\prod_{r=1}^R b_{j_r}\right)\sum_{\substack{\bk\in \mathbb{Z}^n\\\|j_1F^*_{\bj}(\bk/j_1)\|\leq \frac{2^{i+1}}{Q}}}\frac{w_{\bj}^*\kj j_1^{-\frac{n}{2}}}{\sqrt{|\det\, H_{F_{\bj}}\left((\nabla F_{\bj})^{-1}(\bk/j_1)\right)|}}\\
   &\ll \mathcal{E}_n(X)(\log 4X)^{R}X^{-\frac{n}{2}}\left(2^{i+1}Q^{-1} X^{n} +X^{\alst-R}\right).
\end{align*}
Using \eqref{eq EQX} and summing up in the dyadic scales $i$, we get
\begin{align*}
    E(Q, X)&\ll Q^{\frac{n}{2}}\sum_{\substack{1\leq j_1\leq J_1\\ 0\leq j_r\leq X\\ 2\leq r\leq R}}\left(\prod_{r=1}^R b_{j_r}\right)\sum_{\bk\in \mathbb{Z}^n}\frac{w_{\bj}^*\kj}{\sqrt{|\det\, H_{F_{\bj}}\left((\nabla F_{\bj})^{-1}(\bk/j_1)\right)|}}j_1^{-\frac{n}{2}}\min\left(\|j_1F^*_{\bj}(\bk/j_1)\|^{-1}, Q\right)
    \\&+X^{\frac{n}{2}-1}Q^{\frac{n}{2}}\\&\ll \mathcal{E}_n(X)Q^{\frac{n}{2}+1} (\log 4X)^{R}X^{-\frac{n}{2}}\sum_{0\leq i\leq \frac{\log 4Q}{\log 2}}2^{-i}\left(2^{i+1}Q^{-1} X^{n} +X^{\alst-R}\right)\\
    &\ll \mathcal{E}_n(X)(\log 4X)^{R} \left((\log 4Q)Q^{\frac{n}{2}}X^{\frac{n}{2}}+Q^{\frac{n}{2}+1}X^{\alst-\frac{n}{2}-R}\right).
\end{align*}
\end{proof}
We now return to the proof of Theorem \ref{thm homog main}. The estimate in Lemma \ref{lem 2 low bd} for each $E_{r, \bgam}(X, Q)$, independently of $Q$ and $X$. Therefore, \eqref{eq Selb exp} yields
\begin{align*}
    &|\Nwmh-2 \delta^R \mathfrak{N}_0| \\&\ll \delta^{R-1}\frac{1}{X}Q^{n+1}+\frac{1}{X^R} Q^{n+1} +\mathcal{E}_n(X)\left(\log 4X\right)^{R}\left( (\log 4Q)Q^{\frac{n}{2}}X^{\frac{n}{2}}+Q^{\frac{n}{2}+1}X^{\alst-\frac{n}{2}-R}\right).
\end{align*}
Plugging in \eqref{eq No expr}, we get
\begin{align*}
    &\left|\Nwmh- \frac{2\hat{w}(\bzero)}{n+1}\delta^R Q^{n+1}\right| \\&\ll \delta^{R-1}\frac{1}{X}Q^{n+1}+\frac{1}{X^R} Q^{n+1} +\mathcal{E}_n(X)\left(\log 4X\right)^{R}\left( (\log 4Q)Q^{\frac{n}{2}}X^{\frac{n}{2}}+Q^{\frac{n}{2}+1}X^{\alst-\frac{n}{2}-R}\right)+O(Q^{n}).
\end{align*}
We now choose our parameter $X$ and set it to be $Q^{\frac{n+2}{n+2R}}$. This yields
\begin{align*}
    &\left|\Nwmh- \frac{2\hat{w}(\bzero)}{n+1}\delta^R Q^{n+1}\right| \\&\ll \delta^{R-1}Q^{n+1-\frac{n+2}{n+2R}}+Q^{n+1-\frac{(n+2)R}{n+2R}} +\mathcal{E}_n(Q)\left(\log 4Q\right)^{R}\left( (\log 4Q)Q^{\frac{n(n+R+1)}{n+2R}}+Q^{\frac{n}{2}+1}Q^{\left(\frac{n+2}{n+2R}\right)\left(\frac{n(n+R+1)}{n+2}-\frac{n}{2}-R\right)}\right)\\&+O(Q^{n})\\
    &\ll \delta^{R-1}Q^{n+1-\frac{n+2}{n+2R}}+\mathcal{E}_n(Q) Q^{\frac{n(n+R+1)}{n+2R}}
    =\delta^{R-1}Q^{\frac{(n+1)(R-1)}{R}+\frac{n(n+R+1)}{R(n+2R)}}+\mathcal{E}_n(Q) Q^{\frac{n(n+R+1)}{n+2R}},
\end{align*}
with the constants $\mathfrak{c}_2$ (or $\mathfrak{c}_1$) in the definition of $\mathcal{E}_n(Q)$ chosen large enough. 
This establishes \eqref{eq homog main est}, and hence Theorem \ref{thm homog main}.

\section{Proof of Theorem \ref{thm hausd well app}}
\label{sec thm dio approx}
Since $\psi$ is monotonic, by a slight generalisation of Cauchy's condensation test, \eqref{eq Hmeas gset conv} being convergent is equivalent to 
\begin{equation}\label{eq dyad Hmeas}
 \sum_{i=1}^\infty \
 2^{i(n+1)}\left(\frac{\psi_0(2^i)}{2^i}\right)^{s}\prod_{r=1}^R\psi_r(2^i)<\infty\,.
\end{equation}
Let $\bpsi:=\left(\psi_0, \psi_1, \ldots, \psi_r\right)$ and $\mathcal{P}_{\cM}(\bpsi)$ denote the projection of $\cM\cap \mathcal{S}_{n+R}(\bpsi)$ onto $\mathcal{D}\subseteq \mathbb{R}^n$. Since the functions $f_1, \ldots, f_R$ are smooth, this projection is bi-Lipschitz. Consequently,
$$\mathcal{H}^s\left(\cM\cap \mathcal{S}_{n+R}(\boldsymbol{\psi})\right)=0 \iff \mathcal{H}^s\left(\mathcal{P}_{\cM}(\bpsi)\right)=0,$$
and it suffices to show the latter equality.
As $\mathcal{H}^s\left(\mathcal{P}_{\cM}(\bpsi)\right)=0$ for $s>n$, we may assume without loss of generality that $s\leq n$. To be able to use Theorem \ref{thm homog main}, by replacing $\psi_r(q)$ and with $$\max\left(\bpsi_r(q), q^{-\frac{n+2}{n+2R}+\eta}, q^{-\frac{n}{n+2(R-1)-\frac{4}{n}}+\eta}\right),$$ if need be, we may also assume that 
\begin{equation}
    \label{eq assumppsi dio}
    \bpsi_r(q)\geq \max\left(q^{-\frac{n+2}{n+2R}+\eta}, q^{-\frac{n}{n+2(R-1)-\frac{4}{n}}+\eta}\right), \qquad \text{ for all } q\in \mathbb{Z}_{>0} \text{ and } 0\leq r\leq R.
\end{equation}

For $(q, \boldsymbol{p})=(q, \ba, \bob)\in \mathbb{Z}_{>0}\times\mathbb{Z}^{n}\times \mathbb{Z}^{R}$, we define
\begin{equation}
    \sigma(\bop/q):=\left\{\bx\in \mathcal{P}_{\cM}(\bpsi):\left\|\bx-\frac{\ba}{q}\right\|_2\leq \frac{\psi_0(q)}{q}, \left|f_r(\bx)-\frac{b_r}{q}\right|\leq \frac{\psi_r(q)}{q} \text{ for } 1\leq r\leq R\right\}.
\end{equation}
Then $\sigma(\bop/q)$ is contained in a set of radius $\frac{\psi_0(q)}{q}$. Moreover, if $\sigma(\bop/q)\neq \emptyset$, then by the triangle inequality and using the Lipschitz property of $f_r$, we can conclude that
$$\left|f_r\left(\frac{\ba}{q}\right)-\frac{b_r}{q}\right|\ll \frac{\psi_0(q)}{q}+\frac{\psi_r(q)}{q}\ll \frac{\psi_r(q)}{q},$$
with the implicit constant depending only on $f_r$. For each $i\in \mathbb{Z}_{\geq 0}$, we define
\begin{equation}
    \label{def Bi}
    B_i:=\{(q, \bop)\in \mathbb{Z}_{\geq 0}\times \mathbb{Z}^{n+R}: 2^i\leq q< 2^{i+1}, \sigma(\bop/q)\neq \emptyset\}.
\end{equation}
Then we have
\begin{align*}
    \#B_i
    &\leq \#\{\bop/q=(\ba/q, \bob/q)\in \mathbb{Q}^{n+R}: 1\leq q < 2^{i+1}, \ba/q\in \mathscr{D}, |f_r(\ba/q)-b_r/q|\ll \psi_r(q)/q \text{ for } 1\leq r\leq R\}\\
    &\leq \#\{\ba/q\in \mathbb{Q}^{n}: 1\leq q\leq 2^i, \ba/q\in \mathscr{D}, \|f_r(\ba/q)\|\ll \psi_r(q)/q \text{ for } 1\leq r\leq R\}\\
    &\ll \mathfrak{N}_{w, \cM}\left(2^{i+1}, \psi_1(2^i), \ldots, \psi_R(2^i)\right)\ll 2^{i(n+1)}\prod_{r=1}^R\psi_r(2^i) 
\end{align*}
where we used \eqref{eq assumppsi dio} and Corollary \ref{cor hetero} to obtain the last estimate. 
Defining
$$\Sigma_i:=\bigcup_{\bop/q\in B_i} \sigma (\bop/q), $$
we observe that
\begin{equation*}
    \mathcal{P}_{\cM}(\bpsi)=\bigcap_{T\geq 0} \bigcup_{i\geq T} \Sigma_{i}.
\end{equation*}
Now, 
\begin{align*}
    \mathcal{H}^s\left(\bigcup_{i\geq T} \Sigma_i \right)\leq \sum_{i\geq T}\sum_{(q, \bop)\in B_i} \left(\frac{\psi_0(q)}{q}\right)^{s}\ll \sum_{i\geq T} 2^{i(n+1)} \left(\frac{\psi_0(2^i)}{2^i}\right)^{s}\prod_{r=1}^R\psi_r(2^i),
\end{align*}
which converges to $0$ as $T\rightarrow \infty$, due to \eqref{eq dyad Hmeas}. Therefore, we conclude that
$$\mathcal{H}^s\left(\mathcal{P}_{\cM}(\psi)\right)=0.$$

\bibliography{ref}
\end{document}